\documentclass[12pt,reqno]{amsart}
\usepackage[margin=1in]{geometry}
\usepackage{amsmath,amssymb,amsthm,graphicx,amsxtra, setspace}
\usepackage[utf8]{inputenc}
\usepackage{mathrsfs}
\usepackage{hyperref}
\usepackage{upgreek}
\usepackage{mathtools}
\usepackage{hyperref}
\usepackage{alltt}
\usepackage{fouridx}
\usepackage{dsfont}
\usepackage{dsfont}
\usepackage{cancel}
\usepackage{xcolor}
\usepackage[mathcal]{euscript}
\usepackage[margin=1in]{geometry}
\usepackage{shadethm}
\usepackage{float}
\usepackage{todonotes}
\allowdisplaybreaks

\usepackage{todonotes}
\newcommand{\todokk}[1]{\todo[inline,color=green!40]{from KK: #1}}

\usepackage[pagewise]{lineno}
\newtheorem{theorem}{Theorem}[section]
\newtheorem{lemma}[theorem]{Lemma}

\newtheorem{remark}[theorem]{Remark}

\let\originalleft\left
\let\originalright\right
\renewcommand{\left}{\mathopen{}\mathclose\bgroup\originalleft}
\renewcommand{\right}{\aftergroup\egroup\originalright}

\hypersetup{colorlinks=true,%
	citecolor=red,%
	filecolor=blue,%
	linkcolor=blue,%
}
\usepackage{graphicx}

\newcommand{\Addresses}{{
		\footnote{
			
			\noindent \textsuperscript{1}  Southern Mathematical Institute of Vladikavkaz Scientific Center of Russian Academy of Sciences, 362025, Russia. \par\nopagebreak
			\noindent  \textit{e-mail:} \texttt{zhanatotieva@gmail.com}
			
			\noindent \textsuperscript{2} Center for Mathematics and Applications (NovaMath),  NOVA	SST,	Portugal.\par\nopagebreak
			\noindent  \textit{e-mail:} \texttt{kushkinra@gmail.com, k.kinra@fct.unl.pt.}

			\noindent \textsuperscript{3} Department of Mathematics, Indian Institute of Technology Roorkee-IIT Roorkee,
			Haridwar Highway, Roorkee, Uttarakhand 247667, INDIA.\par\nopagebreak
			\noindent  \textit{e-mail:} \texttt{manilfma@iitr.ac.in, maniltmohan@gmail.com.}
            
			\noindent \textsuperscript{*}Corresponding author.

			\textit{Key words:} Wave equation, acoustic boundary condition, inverse problem, memory kernel, contraction mapping principle.
			
			Mathematics Subject Classification (2020): 35R30, 35L05, 35L20, 35Q99.

}}}

\begin{document}
	
	
	\title[]{{\bf INVERSE PROBLEM FOR WAVE EQUATION OF MEMORY  TYPE WITH ACOUSTIC BOUNDARY CONDITIONS: GLOBAL SOLVABILITY}
			\Addresses}
		\author[Zhanna D. Totieva]{Zhanna D. Totieva\textsuperscript{1*}}
	\author[Kush Kinra]{Kush Kinra\textsuperscript{2}}
	\author[Manil T. Mohan]{Manil T. Mohan\textsuperscript{3}}

    	\begin{abstract}
	In this article, we study the one-dimensional inverse problem of determining the memory kernel by the integral overdetermination condition for the direct problem of finding the velocity potential and the displacement of boundary points. A wave equation with initial and acoustic boundary conditions in media with dispersion is used as a mathematical model. The inverse problem is reduced to an equivalent problem with homogeneous boundary conditions for the system of integro-differential equations. Using the technique of estimating integral equations and the contraction mappings principle in Sobolev spaces, the global existence and uniqueness theorem for the inverse problem is proved.
	\end{abstract}

	\maketitle

\section{Introduction}\setcounter{equation}{0}\label{sec-intro}

We consider the following wave equation of memory type with initial and  acoustic boundary conditions for a medium with dispersion:
\begin{equation}\label{eqn-wave}
    \left\{
    \begin{aligned}
u_{tt}-u_{xx}-\beta u_{xxtt}+\int_0^t k(t-s)u_{xx}(x,s)\,ds=0,\ (x,t)\in I\times (0,T),\\
u\big|_{x=0}\equiv 0,\quad  0< t<T,\\
\Big[u_x-\int_0^t k(s)u_{x}(x,t-s)\,ds\Big]\Bigg|_{x=\ell}=y'(t), \quad  0< t<T,\\
u_t\big|_{x=\ell}=-py'(t)-qy(t), \quad  0< t<T,\\
u\big|_{t=0}=u_0(x), \ u_t\big|_{t=0}=u_1(x), \quad  x\in I,
    \end{aligned}
    \right.
\end{equation}
where $u(\cdot,\cdot)$~is~the velocity potential, $k(\cdot)$ is the memory kernel, $y(\cdot)$~is displacement at the point $x=\ell$; $T,\ \beta,\ p,\ q$ are some positive known constants, $I:=(0,\ell)$.

Many authors have recently examined wave equations with acoustic boundary conditions. In \cite{Park+Park_2011}, the problem \eqref{eqn-wave} with $\beta=0$ is considered in the multi-dimensional case, when $I$ is a bounded domain in $\mathbb{R}^n\ (n\geq 1)$ with boundary $\Gamma=\Gamma_0\cup \Gamma_1$ of class $C^2$, $\Gamma_0$ and $\Gamma_1$ are closed and disjoint. In equation $\eqref{eqn-wave}_3$, the derivative with respect to the unit outward normal to $\Gamma$ is used. The work \cite{Park+Park_2011}  investigated the influence of the kernel function $k$ and demonstrated general decay rates of solutions when $k$ does not decay exponentially. The authors in \cite{Li+Xi_2019} considered the nonlinear viscoelastic Kirchhoff-type equation with initial conditions and acoustic boundary conditions. They showed that, depending on the properties of the convolution kernel $k$ at infinity, the energy of the solution decays exponentially or polynomially as $t\rightarrow +\infty.$ The approach is based on integral inequalities and multiplier techniques. Instead of using a Lyapunov-type technique for some perturbed energy, the authors of the work \cite{Li+Xi_2019} concentrated on the original energy by showing that it satisfies a nonlinear integral inequality which, in turn, yields the final decay estimate.

Boundary conditions of memory type, as  $\eqref{eqn-wave}_3$, imposed
on a portion of the boundary and Dirichlet condition on the other part  of the boundary,
have been considered, for example in   \cite{Park+Kang_2008,Santos_2001,Vicente_2009}. Condition  $\eqref{eqn-wave}_3$ means that
the right end of $I$ is clamped in a body with viscoelastic properties. For a detailed analysis in this direction, the interested  readers are referred to the articles \cite{Li+Xi_2019,Park+Park_2011}.

The kernel of memory term in $\eqref{eqn-wave}_1$ represents the physical properties of a viscoelasticity medium, which is difficult to determine directly. By taking additional information on the velocity potential $u$, we can reconstruct $k$, provided the velocity potential is taken on a suitable subset of  $I.$ We suppose that such an additional information on $u$ can be represented in an integral form, called \emph{integral overdetermination condition}, by
\begin{align}\label{eqn-integral-overdetermination}
    \int_0^\ell  (\varphi(x)-\beta \varphi''(x))u_{x}(x,t)dx=f(t),\quad t\in [0,T],
\end{align}
where $f(t),t\in [0,T],$ is the measurement data, which represents the average velocity on $I$ and $\varphi$ is a given function representing the type of device used to measure the velocity $u_x$. Further, we can consider $\varphi$ as an internal tiny sensor that measures the average velocity $u_x$ in a very small subset of the domain $I$, that is,  $I':=\mathrm{supp}(\varphi)\subset I $ may be very small. The physical phenomena behind this mathematical issue shows the  averaging effect of measurement data, where both the convolution $\int_0^tk(t-s) u_{xx}(x,s) d s$ and the integral $\int_0^\ell  (\varphi(x)-\beta \varphi''(x))u_{x}(x,t)dx$  stands for some averaging process (cf. \cite{POV,WL}, etc. for more details).

If function $k(\cdot)$ is known, then the initial-boundary value problem \eqref{eqn-wave} of determining $u(\cdot,\cdot)$ and $y(\cdot)$ is called {\it the direct problem}.  {\it The inverse problem} is the problem of determining $u(\cdot,\cdot),\ y(\cdot)$ and $k(\cdot)$ from a system of partial integro-differential equations \eqref{eqn-wave}--\eqref{eqn-integral-overdetermination}. This study relates to the class of one-dimensional inverse problems in the linear dynamical wave processes of memory type. The unknown function in the above problem is the kernel of the integral operator modelling the memory
phenomenon, which is relevant when wave processes propagate in media.

In the theory of inverse problems for a hyperbolic integro-differential equations, those of finding the integral-operator kernel (dependent on time and spatial variables) comprise a field that shaped at the end of the last century \cite{Grasselli+Kabanikhin+Lorenzi_1992,Cavaterra+Grasselli_1994,Bukhgeim_1993,Durdiev_1994,Bukhgeim+Dyatlov_1998}.
A more detailed analysis of the literature on this subject is available in \cite{Durdiev+Totieva_2023}, which is one of the most recent fundamental papers in the field of inverse problems for media with after-effect. It contains the results of studying the well-posedness of a number of one- and
multi-dimensional inverse dynamical problems for hyperbolic integro-differential equations that occur when internal characteristics of media with after-effects are described using measurements of the wave field in accessible domains.

The study of multi-dimensional inverse problems is very limited in the literature. Examining such problems is of both mathematical and application interest. The problems of determining kernel which depends on two or more variables are considered in \cite{Durdiev_2020,Janno+Wolfersdorf_2001,Karchevsky+Fatianov_2001,Lorenzi+Messina+Romanov_2007,Lorenzi+Romanov_2011,Romanov_2012,Romanov_2012_JAIM,Romanov_2014,Romanov+Yamamoto_2010,Totieva_2024}. In \cite{Janno+Wolfersdorf_2001}, the method of separation of variables is used to solve inverse problems in a bounded domain, by which the problems are reduced to a system of integral equations of the Volterra type with respect to unknown functions depending on a time variable. In \cite{Lorenzi+Messina+Romanov_2007,Lorenzi+Romanov_2011,Romanov_2012,Romanov_2012_JAIM,Romanov_2014,Romanov+Yamamoto_2010}, multi-dimensional inverse problems with concentrated sources of perturbation are studied. These problems are reduced to solving problems of integral geometry. The purposes of inverse problems is to determine the spatial parts of kernel and Lame coefficients.

Later in \cite{Totieva_2024}, using the method of scales of Banach spaces for the viscoelasticity equation (in the class of functions analytic
in the variable $x$ and smooth in the variable $t$)  in combination with the method of weighted norms, the global unique solvability of the kernel-determining problem was studied. In recent years, we have witnessed an increase in the number of works on numerical computations of integral operator kernels \cite{Bozorov_2020,Karchevsky+Fatianov_2001,Durdiev_2020,Kaltenbacher+Khristenko+Nikolic+Rajendran+Wohlmuth_2022}.

We also mention that Colombo et al. examined some inverse problems of reconstructing the kernel of memory term for certain scalar parabolic equations and damped wave equations with a memory term in \cite{Colombo+Lorenzi_1998_ADE,Colombo+Lorenzi_1997_JMAA,Colombo_2001_DIE,Colombo_2001_JMAA,Colombo+Guidetti+Lorenzi_2003_AMSA,Colombo+Guidetti+Lorenzi_2003_DSA,Colombo+Guidetti_2002_ZAA,Colombo_2005,Colombo_2007_PD,Colombo_2007_Nonlinearity} and many others. In \cite{Colombo+Guidetti_2007}, Colombo and Guidetti  addressed the results of the uniqueness of solutions and the presence of both local- and global-in-time for an evolution equation with nonlinearity and acceptable growth condition. In their novel approach, the authors rewrite the convolution term as the sum of two linear terms in the unknown. 

The theoretical significance of current study is to obtain the necessary and sufficient conditions for the global unique solvability of the one-dimensional inverse problem \eqref{eqn-wave}--\eqref{eqn-integral-overdetermination}. For the case of $\beta=0$, the result  on local-in-time existence and uniqueness of the kernel determining problem \eqref{eqn-wave}--\eqref{eqn-integral-overdetermination} was obtained in \cite{Totieva+Durdiev_2024}. In the literature (see \cite[Chapter 11]{Whitham_1999}) it has been noticed that the wave processes of a medium with dispersion are described by the system \eqref{eqn-wave} which is the main motivation for considering this work. Equation \eqref{eqn-wave} with $k=0$ appears in the elasticity of longitudinal waves in bars, in water waves in the Boussinesq approximation for long waves, and in plasma waves. The term $u_{xxtt}$ can be replaced in some cases by $u_{xxxx}.$ It depends on the dispersion relation \cite[Chapter 11]{Whitham_1999}. However, for a more accurate description of wave processes in media with dispersion, a kernel $k$ is needed. To the best of our knowledge, no one has explored this type of problem. 

It is more convenient to study the inverse problem \eqref{eqn-wave}--\eqref{eqn-integral-overdetermination} in terms of the functions $$v:=u_t+z\frac x\ell,\ z:=py'+qy.$$ Equations \eqref{eqn-wave}--\eqref{eqn-integral-overdetermination} in terms of functions $v$, $z$ are given in Subsection \ref{Inverse-problem}. The function $v$ has zero boundary conditions, which is necessary when applying the energy inequality in Section \ref{sec-main-result}.

The organization of this article is as follows: In the next section, we provide necessary function spaces which help us to present a concrete version of the inverse problem. Section \ref{sec-equiv-prob} is devoted to the discussion of an equivalent problem corresponding to the original inverse problem. In the final section, we establish the main result of this article, that is, the global-in-time existence and uniqueness of solutions to the inverse problem corresponding to the system \eqref{eqn-wave}--\eqref{eqn-integral-overdetermination} (see Theorem \ref{thm-4.3}). In the appendix, we discuss the energy estimates for a linear problem corresponding to \eqref{eqn-wave} that have been used to demonstrate the main result of this article.

\section{The concrete version of the inverse problem}\label{sec-inverse-problem}\setcounter{equation}{0}

\subsection{Function spaces} For any integers $m,p$ we denote by $W^{m,p}(I):=W^{m,p}(I;\mathbb{R})$ and $W^{m,p}(0,T):=W^{m,p}(0,T;\mathbb{R})$ for the usual Sobolev spaces defined for spatial variable and time variable, respectively. For a Banach space $\mathrm{X},$ the space $L^p(0,T;\mathrm{X})$ consists the equivalence class of all Lebesgue measurable functions $u:[0,T]\rightarrow \mathrm{X}$ with
$$
\|u\|_{L^p(0,T;\mathrm{X})}:=\left(\int_0^T\|u(t)\|_{\mathrm{X}}^p\,dt\right)^{1/p}<\infty,\quad 1\leq p<\infty.
$$

The Sobolev space $W^{m,p}(0,T;\mathrm{X})$ consists all functions $u\in L^p(0,T;\mathrm{X})$ such that $\frac {\partial^{\gamma} u}{\partial t^{\gamma}}$ exists in the weak sense and belongs to $L^p(0,T;\mathrm{X})$ for all $0\leq \gamma \leq m.$ Let $H^m(I):=W^{m,2}(I),\ H^m(0,T):=W^{m,2}(0,T),$ and
$H^m(0,T;\mathrm{X}):=W^{m,2}(0,T;\mathrm{X}).$ We define $a\wedge b:=\min\{a,b\},\ a \vee b:=\max\{a,b\}.$

As $H^{s}(I)\subset C^k(\overline{I})$, for $s>\frac{1}{2}+k,$ $k=0,1,2\ldots,$ let us represent
\begin{align*}H_0^1(I)&=\{\varphi\in H^1(I): \varphi(0)=\varphi(\ell)=0\},\\
H_0^3(I)&=\{\varphi\in H^3(I): \varphi(0)=\varphi(\ell)=0, \varphi'(0)=\varphi'(\ell)=0,\varphi''(0)=\varphi''(\ell)=0\}.
\end{align*}

\subsection{The inverse problem}\label{Inverse-problem} For $T>0$, the inverse problem is to determine 
\begin{align}\label{eqn-2.2}
    v\in H^2(0,\tau;H^1_0(I)\cap H^2(I)),\  k\in H^1(0,\tau), \text{ and } y\in H^3(0,\tau), \ \tau\in (0,T],
\end{align}
such that $(v,k,y)$ satisfies the system:
\begin{equation}\label{eqn-wave-equiv}
    \left\{
    \begin{aligned}
v_{tt}-v_{xx}-\beta v_{xxtt}+k(t)u''_0(x) & +\int_0^t k(t-s)v_{xx}(x,s)\,ds=z''\frac x\ell,\  && \hspace{-30mm}  (x,t)\in I\times (0,\tau),\\ \\
v\big|_{x=0}&=v\big|_{x=\ell}=0,\quad && \hspace{-30mm} 0\leq t<\tau,\\ \\
\Big[v_x-\int_0^t k(t-s)v_x(x,s)\,ds\Big]_{x=\ell} & =y''(t)+k(t)u'_0(\ell) +\frac 1\ell\left(z(t)-\int_0^tk(t-s)z(s)\,ds\right), \\&\quad \quad && \hspace{-30mm}  0< t<\tau,\\ \\
v\big|_{t=0} & =v_0(x), \ v_t\big|_{t=0}=v_1(x),  \quad && \hspace{-30mm} x\in I,\\ \\
\int_0^\ell (\varphi(x)-\beta \varphi''(x)) & \left[v_x-\frac {z}{\ell}\right]dx=f'(t),\quad && \hspace{-30mm} t\in (0,\tau), 
    \end{aligned}
    \right.
\end{equation}
under the following assumption on the data:
\begin{itemize}
    \item [$(i1)$]$u_0(\cdot), u_1(\cdot)\in H^4(I), \varphi(\cdot)\in H_0^3(I).$
    \item [$(i2)$]  $v_0(x):=u_1(x)-u_1(\ell)\frac x\ell, v_1(x):=u_2(x)-u_2(\ell)\frac x\ell  \text{ for all } x\in I,$  where $u_2(\cdot):=\left(1-\beta \frac{d^2}{dx^2}\right)^{-1}u''_0(\cdot)$,  and $v_0(\cdot),v_1(\cdot)\in H^1_0(I)\cap H^2(I).$
     
    \item [$(i3)$] $\alpha^{-1}:=\int_0^\ell \varphi'(x)u''_0(x)dx\neq 0.$
    \item  [$(i4)$] $f(\cdot)\in H^4(0,T):$
    \begin{align*}
        &\int_0^\ell (\varphi(x)-\beta \varphi''(x))u'_0(x)dx =f(0),\\
       & \int_0^\ell (\varphi(x)-\beta \varphi''(x))u'_1(x)dx  =f'(0),\\
       & \int_0^\ell (\varphi(x)-\beta \varphi''(x))u'_2(x)dx  =f''(0),\\
       &\int_0^\ell (\varphi(x)-\beta \varphi''(x))\left(1-\beta \frac{d^2}{dx^2}\right)^{-1}\left[u'''_1(x)-k(0)u'''_0(x)\right]dx   =f'''(0),
    \end{align*}
with $k(0)=\alpha(f'''(0)-\int_0^\ell v_0(x)\varphi'''(x)dx).$
\item [$(i5)$] $y''(0)=\frac{1}{\psi(\ell)}\left[f'(0)-k(0)f(0)+\int_0^\ell \psi(x)(u''_1(x)-k(0)u''_0(x))dx\right]$\\
with  $\psi(x)=\int_0^x\varphi(x)dx-\beta \varphi'(x)$ and $\psi(\ell)\neq 0.$
\end{itemize}

\smallskip

\begin{remark}\label{Rem-2.1}
    By Sobolev's  embedding theorem, for $j+1/2<s,$  if $ h(\cdot)\in H^s(0,T),$ then $h(\cdot) \in C^j[0,T]$ and from this the equalities in  $(i4)$ and $(i5)$ make sense.
\end{remark}

\begin{remark}\label{Rem-2.2}
    Overdetermination condition \eqref{eqn-integral-overdetermination} may be rewritten as:
    \begin{align}\label{eqn-integral-overdetermination-equiv-1}
      -\int_0^\ell  (\varphi'(x)-\beta\varphi'''(x))u(x,t)dx=f(t),\quad t\in (0,\tau)  
    \end{align}
or
\begin{align}\label{eqn-integral-overdetermination-equiv-2}
    \psi(\ell)u_x(\ell,t)-\int_0^\ell \psi(x)u_{xx}(x)dx=f(t), 
\end{align}
where $\psi(x)$ is given in $(i5).$

{Then we deduce  the equality in $\eqref{eqn-wave-equiv}_3$ using \eqref{eqn-integral-overdetermination-equiv-2} as
\begin{align}\label{eqn-2.6}
    y''(t)=G[v_{xx}](t)- k(t)\widehat{G}[u_{xx}](0) -\int_0^tk(t-s)G[v_{xx}](s)ds,
\end{align}
where $$G[\cdot](t):=\frac{1}{\psi(\ell)}\left(f'(t)+\int_0^\ell \psi(x)[\cdot](x,t)dx\right)$$ 
and 
$$\widehat{G}[\cdot](t):=\frac{1}{\psi(\ell)}\left(f(t)+\int_0^\ell \psi(x)[\cdot](x,t)dx\right).$$}
\end{remark}

\subsection{Important inequalities and results} The following inequalities are frequently used in this paper. 

\begin{lemma}[{\cite[Theorem 4.4]{Colombo+Guidetti_2007}}]\label{lem-2.1}
    Let $X$ be a Banach space, $p \in (1, \infty),\ \tau\in \mathbb{R}_+$, $k\in L^p(0,\tau)$, and $f\in L^p(0,\tau; X)$. Then $k \ast f \in L^p(0, \tau ; X)$ and
$$
\|k\ast f\|_{L^p(0, \tau ; X)}\leq \tau^{1-1/p}\|k\|_{L^p(0, \tau)}\|f\|_{L^p(0, \tau ; X)},
$$
where $(k \ast f)(t) = \int_0^t k(t-s)f(s)ds$.
\end{lemma}
\begin{proof} Using Young's inequality for convolution and H\"{o}lder's inequality, we have $$\|k\ast f\|_{L^p(0,\tau; X)} \leq \|k\|_{L^1(0,\tau)}\|f\|_{L^p(0, \tau; X)}\leq  \tau^{1-1/p}\|k\|_{L^p(0, \tau)}\|f\|_{L^p(0, \tau ; X)},$$ which completes the proof.\end{proof}
\begin{lemma}[{\cite[Theorem 4.5]{Colombo+Guidetti_2007}}]\label{lem-2.2}
    Let $X$ be a Banach space, $p \in (1, \infty), \tau\in \mathbb{R_+}$,  $w\in W^{1,p}(0, \tau ; X)$ with $w(0) = 0$. Then
    \begin{align*}
\|w\|_{L^{\infty}(0,\tau; X)} & \leq \tau^{1-1/p}\|w_t\|_{L^p(0, \tau; X)},
\\
\|w\|_{L^{p}(0,\tau; X)} & \leq \tau\|w_t\|_{L^p(0, \tau; X)}.
    \end{align*}
\end{lemma}

The proof can be easily concluded from H\"{o}lder's inequality and Young's inequality for
convolution, using the formula $w = 1\ast w_t$ with $w(0)=0$.

\section{The Equivalent Problem}\label{sec-equiv-prob}\setcounter{equation}{0}

{\begin{lemma}\label{lem-3.1}
    Let the assumptions $(i1)$-$(i5)$ hold. Let $(u,k,y)$ be a solution of the system \eqref{eqn-wave}--\eqref{eqn-integral-overdetermination} such that
    \begin{align}\label{eqn-u-regularity}
        u\in H^3(0,T; H^2(I)),\  k\in H^1(0,T), \text{ and } y\in H^3(0,T). 
    \end{align}
Then $(v:=u_t+z\frac{x}{\ell},k,y)$ satisfies
\begin{align}\label{eqn-3.1}
    v\in H^2(0,T;H_0^1(I)\cap H^2(I)),\  k\in H^1(0,T),\ y\in H^3(0,T)
\end{align}
and solves the system
\begin{align}
v_{tt}-v_{xx}-\beta v_{xxtt}+k(t)u''_0(x) & +\int_0^t k(t-s)v_{xx}(x,s)\,ds=z''\frac x\ell,\ &&(x,t)\in I\times (0,T),\label{eqn-3.2}\\
v\big|_{x=0} & =v\big|_{x=\ell}=0, \quad && 0<t<\tau, \label{eqn-3.3}\\
v\big|_{t=0} & =v_0(x), \ v_t\big|_{t=0}=v_1(x), \quad && x\in I, \label{eqn-3.4}
\end{align}
with
\begin{align}
    k'(t)& =\alpha\bigg\{f^{(iv)}(t)+\int_0^\ell v_t(x,t)\varphi'''(x)dx
-k(0)\int_0^\ell v(x,t)\varphi'''(x)dx
\nonumber \\ & \quad 
-\int_0^t\int_0^\ell k'(t-s) v(x,s)\varphi'''(x)dxds\bigg\}, \label{eqn-3.5}\\
    y'''(t) & =G'[v_{xx}](t)- k'(t) \widehat{G}[u_{xx}](0)
-k(0)G[v_{xx}](t)
 -\int_0^tk'(t-s)G[v_{xx}](s)\,ds, \label{eqn-3.6}
\end{align}
for $ 0< t<T$, where $z''=py'''+qy''.$

On the other hand, if $(v,k,y)$ satisfies \eqref{eqn-3.1} and is a solution to the system \eqref{eqn-3.2}--\eqref{eqn-3.6}, then under the setting $u(t):= u_0+ \int_0^t [v(s)-z(s)\frac{x}{\ell}]ds$, $(u,k,y)$ satisfies \eqref{eqn-u-regularity} and is a solution to the system \eqref{eqn-wave}--\eqref{eqn-integral-overdetermination}.
\end{lemma}}
\begin{proof}
    {\bf Step 1.} Suppose that the system \eqref{eqn-wave}--\eqref{eqn-integral-overdetermination} has a solution $(u,k,y)$ satisfying \eqref{eqn-u-regularity}.
Then, it is immediate to see that $(v,k,y)$ satisfies condition \eqref{eqn-3.1} and equations \eqref{eqn-3.2}, \eqref{eqn-3.3}.

From the equations $\eqref{eqn-wave}_1$, $\eqref{eqn-wave}_5$, we obtain for all $x\in I$
\begin{align*}
    v(x,0)& =u_1(x)-u_1(\ell)\frac x\ell=v_0(x),
\\
 u_{tt}(x,0) & =u_2(x),
\\
 v_t(x,0) & =u_2(x)-u_2(\ell)\frac x\ell=v_1(x).
\end{align*}
Therefore, we have \eqref{eqn-3.4}.

Taking the inner product in \eqref{eqn-3.2} with $\varphi'$, integrating by parts, and using assumptions $(i1)$ and $(i3)$, we get
\begin{align}\label{eqn-3.7}
    k(t)=\alpha\Bigg\{f'''(t)+\int_0^\ell v(x,t)\varphi'''(x)dx
-\int_0^t\int_0^\ell k(t-s) v(x,s)\varphi'''(x)dxds\Bigg\},
\end{align}
where $f'''(t)=-\int_0^\ell (\varphi'(x)-\beta \varphi'''(x)\left(v_{tt}-z''\frac x\ell\right)dx$ and after differentiating by $t,$ the equation \eqref{eqn-3.5} follows.

Using \eqref{eqn-2.6} in Remark \ref{Rem-2.2} and then differentiating it by $t$, the equation \eqref{eqn-3.6} follows.

\vskip 2mm
\noindent
{\bf Step 2.} Suppose that the system \eqref{eqn-3.2}--\eqref{eqn-3.6} has a solution $(v,k,y)$ satisfying \eqref{eqn-3.1}.  It can be easily shown that \eqref{eqn-u-regularity} and $\eqref{eqn-wave}_2$--$\eqref{eqn-wave}_5$ hold. Since $v=u_t+z\frac x\ell$, the equation \eqref{eqn-3.2}  can be rewritten as
\begin{align*}
    \frac{\partial}{\partial t}\left(u_{tt}+z'\frac x\ell-u_{xx}-\beta u_{xxtt}-\int_0^tk(s)u_{xx}(x,t-s)ds\right)=z''\frac x\ell.
\end{align*}
Integrating the above equation, we obtain
\begin{align*}
    u_{tt}+z'\frac x\ell-u_{xx}-\beta u_{xxtt}-\int_0^tk(t-s)u_{xx}ds=z'\frac x\ell+C_1,
\end{align*}
where $C_1$ is a constant. Taking $t=0$, we have $C_1=0$. Hence, we obtain $\eqref{eqn-wave}_1$.

The equation \eqref{eqn-3.5} for $k'$ can be rewritten as
\begin{align*}
    f^{(iv)}(t)&=\frac{\partial}{\partial t}\Biggl(-\int_0^\ell u_{xxt}(x,t)\varphi'(x)dx+k(0)\int_0^\ell u_{xx}\varphi'(x)dx
\\
& \quad +\int_0^t\int_0^\ell k'(t-s)\varphi'(x)u_{xx}(x,s)dxds\Biggr).
\end{align*}
Integrating this equation with respect to $t$, we find
\begin{align*}
    f'''(t) & =-\int_0^\ell u_{xxt}(x,t)\varphi'(x)dx+k(0)\int_0^\ell u_{xx}\varphi'(x)dx \\ 
& \quad +\int_0^t\int_0^{\ell}k'(t-s)\varphi'(x)u_{xx}(x,s)dxds+C_2,
\end{align*}
where $C_2$ is a constant. Taking $t=0$, and using assumption $(i4)$, we obtain $C_2=0$.
At the same time
\begin{align*}
    f'''(t)=\frac{\partial}{\partial t}\left(-\int_0^\ell u_{xx}(x,t)\varphi'(x)dx+\int_0^t\int_0^{\ell}k(t-s)\varphi'(x)u_{xx}(x,s)dxds\right).
\end{align*}
Integrating with respect to $t,$ we have
\begin{align*}
    f''(t)=-\int_0^\ell u_{xx}(x,t)\varphi'(x)dx+\int_0^t\int_0^\ell k(t-s)\varphi'(x)u_{xx}(x,s)dxds+C_3.
\end{align*}
Taking $t=0$, and using  assumption $(i4)$, we get $C_3=0$, so that 
\begin{align*}
    f''(t)=-\int_0^\ell \varphi'(x)u_{xx}(x,t)dx+\int_0^t\int_0^\ell k(t-s)\varphi'(x)u_{xx}(x,s)dxds.
\end{align*}
Using $\eqref{eqn-wave}_1$ in the above equation along with assumption $(i1)$, we have
\begin{align*}
    f''(t)=-\int_0^\ell \varphi'(x)(u_{tt}(x,t)-\beta u_{xxtt}(x,t))dx=-\frac{\partial}{\partial t}\int_0^\ell \varphi'(x)(u_{t}(x,t)-\beta u_{xxt}(x,t))dx.
\end{align*}
Integrating twice the above equation, taking $t=0$ and using assumption $(i4)$, we find
\begin{align*}
    f(t)=-\int_0^\ell \varphi'(x)(u(x,t)-\beta u_{xx}(x,t))dx=-\int_0^\ell \varphi'(x)u(x,t)dx+\int_0^\ell \beta \varphi'''(x)u(x,t)dx,
\end{align*}
which is \eqref{eqn-integral-overdetermination-equiv-1} and hence \eqref{eqn-integral-overdetermination}.

The equivalence of equalities \eqref{eqn-3.6} and \eqref{eqn-wave-equiv} is proved in a similar way using  assumption $(i5)$. This completes the proof of Lemma \ref{lem-3.1}.
\end{proof}

\section{The Main Result}\label{sec-main-result}\setcounter{equation}{0}

From the equivalent problem \eqref{eqn-3.1}--\eqref{eqn-3.6}, we see that the unknown function $y'''(t)$ can easily be excluded from the system \eqref{eqn-3.1}--\eqref{eqn-3.6}, since it is expressed in terms of $k$ and $v.$ Therefore, the main result will be proved for functions $k,v$.
\medskip

The global solvability  technique used in this section is adapted from \cite{Kumar+Kinra+Mohan_2021}.
\begin{theorem}[Local-in-time existence]\label{thm-4.1}
    Let assumptions $(i1)$-$(i5)$ hold. Then there exist $\tau \in (0,T)$, such that the inverse problem \eqref{eqn-3.1}--\eqref{eqn-3.6} has a unique solution 
    \begin{align*}
    (v,k)\in H^2(0,\tau;H_0^1(I)\cap H^2(I))\times H^1(0,\tau).
\end{align*}
\end{theorem}

\begin{proof}[Proof of Theorem \ref{thm-4.1}]
 Let $Q(\tau,M)$ be the space of functions
 $(\tilde v,\tilde k')\in H^2(0,\tau;H_0^1(I)\cap H^2(I))\times L^2(0,\tau)$ such that
\begin{align*}
    \|\tilde v\|_{H^2(0,\tau;H_0^1(I)\cap H^2(I))}+\|\tilde k'\|_{L^2(0,\tau)}\leq M,
\end{align*}
 and \eqref{eqn-3.3}--\eqref{eqn-3.6} hold. The constant $M>0$ will be determined later.

We define the mapping $A: Q(\tau, M)\rightarrow  Q(\tau, M)$ such that $(\tilde v,\tilde k')\rightarrow  (v,k')$ through
\begin{align}
    k'(t) & =\alpha\Bigg\{f^{(iv)}(t)-\int_0^\ell  \tilde v_t(x,t)\varphi'''(x)dx
-k(0)\int_0^\ell \tilde v(x,t)\varphi'''(x)dx
\nonumber  \\
 & \qquad \quad -\int_0^t\int_0^\ell  \tilde k'(t-s) \tilde v(x,s)\varphi'''(x)dxds\Bigg\},
\label{eqn-4.1}
\end{align}
and the initial boundary value problem
\begin{equation}\label{eqn-4.2}
    \left\{
    \begin{aligned}
        v_{tt}-v_{xx}-\beta v_{xxtt}+\tilde k(t)u''_0(x)+\int_0^t \tilde k(t-s)\tilde v_{xx}(x,s)\,ds & =z''\frac x\ell, && (x,t)\in I\times (0,T),
\\
v\big|_{x=0}=v\big|_{x=\ell} & =0,  && 0<t<T,
\\
v\big|_{t=0}=v_0(x), \ v_t\big|_{t=0} & =v_1(x),   &&  x\in I,
    \end{aligned}
    \right.
\end{equation}
with
\begin{align}\label{eqn-4.3}
    y'''(t) & =G'[\tilde v_{xx}](t)  -  k'(t) \widehat{G}[u_{xx}](0)-k(0)G[\tilde v_{xx}](t)
 -\int_0^t\tilde k'(t-s)G[\tilde v_{xx}](s)\,ds, 
\end{align}
for $ 0< t<T.$ In the above equalities $\tilde k (t)=k(0)+\int_0^t \tilde k'(s)ds.$ 

Our aim is to show that the mapping $A: Q(\tau, M)\rightarrow  Q(\tau, M)$ is a \emph{contraction map.}

\vskip 2mm
\noindent
{\bf Step 1.} Firstly, we show that the map $A$ is well-defined for an appropriate choice of $M$ and $\tau.$
Therefore from \eqref{eqn-4.1}, we find
\begin{equation}\label{eqn-4.4}
    \begin{aligned}
       & \|k'\|_{L^2(0,\tau)}
       \\ & \leq \alpha \Bigg\{\|f^{(iv)}(t)\|_{L^2(0,\tau)}+\|\tilde v_t\|_{L^2(0,\tau;L^2(I))}\|\varphi'''\|_{L^2(I)}
+|k(0)|\|\tilde v\|_{L^2(0,\tau;L^2(I))}\|\varphi'''\|_{L^2(I)}
\\ & \qquad  +\left\|\int_0^t\int_0^\ell \tilde k'(t-s) \tilde v(x,s)\varphi'''(x)dxds\right\|_{L^2(0,\tau)}\Bigg\}.
    \end{aligned}
\end{equation}

Using Lemma \ref{lem-2.1} for $p=2$ and $X=L^2(I)$,
we get
\begin{align}\label{eqn-4.5}
\left\|\int_0^t\int_0^\ell \tilde k'(t-s) \tilde v(x,s)\varphi'''(x)dxds\right\|_{L^2(0,\tau)} 
         & = \left\|\int_0^\ell  [\tilde k'\ast \tilde v(x)]\varphi'''(x)dx\right\|_{L^2(0,\tau)} 
        \nonumber\\ & \leq \left\|k'\ast \tilde v\right\|_{L^2(0,\tau;L^2(I))}\|\varphi'''\|_{L^2(I)}
        \nonumber\\ & \leq \tau^{1/2}\|\tilde k'\|_{L^2(0,\tau)}\|\varphi'''\|_{L^2(I)}\|\tilde v\|_{L^2(0,\tau;L^2(I))}.
\end{align}

An application of Lemma \ref{lem-2.2} implies
\begin{align}
        \|\tilde v\|_{L^2(0,\tau;L^2(I))} & \leq \|\tilde v-v_0\|_{L^2(0,\tau;L^2(I))}+\|v_0\|_{L^2(0,\tau;L^2(I))}
\nonumber\\ 
& \leq\tau \|(\tilde v-v_0)_t\|_{L^2(0,\tau;L^2(I))}+\tau^{1/2}\|v_0\|_{L^2(I)}
\nonumber\\
& =\tau \|\tilde v_t\|_{L^2(0,\tau;L^2(I))}+\tau^{1/2}\|v_0\|_{L^2(I)}
\nonumber\\ & \leq \tau (\tau\|\tilde v_{tt}\|_{L^2(0,\tau;L^2(I))}+\tau^{1/2}\|v_1\|_{L^2(I)})
 +\tau^{1/2}\|v_0\|_{L^2(I)}
\nonumber\\ & =\tau^2\|\tilde v_{tt}\|_{L^2(0,\tau;L^2(I))}+\tau^{3/2}\|v_1\|_{L^2(I)}+\tau^{1/2}\|v_0\|_{L^2(I)},\label{eqn-4.6}\\
        \|\tilde v_t\|_{L^2(0,\tau;L^2(I))} & \leq \|\tilde v_t-v_1\|_{L^2(0,\tau;L^2(I))}+\|v_1\|_{L^2(0,\tau;L^2(I))}
\nonumber\\ & 
\leq\tau \|(\tilde v_t-v_1)_t\|_{L^2(0,\tau;L^2(I))}+\tau^{1/2}\|v_1\|_{L^2(I)}
\nonumber\\ & 
=\tau \|\tilde v_{tt}\|_{L^2(0,\tau;L^2(I))}+\tau^{1/2}\|v_1\|_{L^2(I)}.\label{eqn-4.7}
\end{align}

\medskip
Using \eqref{eqn-4.5}-\eqref{eqn-4.7} in \eqref{eqn-4.4} and then by applying the definition of the space $Q(\tau,M)$, we arrive at:
\begin{align}\label{eqn-4.8}
&\|k'\|_{L^2(0,\tau)}
\nonumber\\ & \leq \alpha \Biggr\{\|f^{(iv)}\|_{L^2(0,\tau)}+\|\varphi'''\|_{L^2(I)}\Big[\tau\|\tilde v_{tt}\|_{L^2(0,\tau;L^2(I))}+\tau^{1/2}\|v_1\|_{L^2(I)}
\nonumber \\ &
\quad +(\tau^2\|\tilde v_{tt}\|_{L^2(0,\tau;L^2(I))}+\tau^{3/2}\|v_1\|_{L^2(I)}+\tau^{1/2}\|v_0\|_{L^2(I)})(|k(0)|+\tau^{1/2}\|\tilde k'\|_{L^2(0,\tau)})\Big]\Biggr\}
\nonumber \\ &
\leq \alpha \Biggr\{\|f^{(iv)}\|_{L^2(0,\tau)}+\|\varphi'''\|_{L^2(I)}\Big[\tau M+\tau^{1/2}\|v_1\|_{L^2(I)}
\nonumber \\ &
\quad +(\tau^2M+\tau^{3/2}\|v_1\|_{L^2(I)}+\tau^{1/2}\|v_0\|_{L^2(I)})(|k(0)|+\tau^{1/2}M)\Big]\Biggr\}.
\end{align}
Next we estimate $y'''$ from \eqref{eqn-4.3} as
\begin{align}\label{eqn-4.9}
& \|y'''\|_{L^2(0,\tau)}
 \nonumber\\ &\leq \|G'[\tilde v_{xx}]\|_{L^2(0,\tau)}+|\widehat{G}[u_{xx}](0)|\|k'\|_{L^2(0,\tau)}
+|k(0)|\|G[\tilde v_{xx}]\|_{L^2(0,\tau)}
\nonumber\\ & 
\quad +\tau^{1/2}\|\tilde k'\|_{L^2(0,\tau)}\|G[v_{xx}]\|_{L^2(0,\tau)}
\nonumber\\ & \leq \frac{1}{|\psi(\ell)|}\Big(\|f''\|_{L^2(0,\tau)}
 +\|\psi\|_{L^2(I)}(\tau \|\tilde v\|_{H^2(0,\tau;H^2(I))}+\tau^{1/2}\|v_1\|_{L^2(I)})\Big)
 \nonumber\\ & 
\quad +|\widehat{G}[u_{xx}](0)|\|k'\|_{L^2(0,\tau)}
 +\frac{1}{|\psi(\ell)|}(|k(0)|+\tau^{1/2}\|\tilde k'\|_{L^2(0,\tau)}) \Big(\|f'\|_{L^2(0,\tau)}
\nonumber\\ & 
\quad +\|\psi\|_{L^2(I)}(\tau^2\|\tilde v\|_{H^2(0,\tau;H^2(I))}+\tau^{3/2}\|v_1\|_{L^2(I)}+\tau^{1/2}\|v_0\|_{L^2(I)})\Big)
\nonumber\\ & \leq  \frac{1}{|\psi(\ell)|}\Big(\|f''\|_{L^2(0,\tau)}|
+\|\psi\|_{L^2(I)}(\tau M+\tau^{1/2}\|v_1\|_{L^2(I)})\Big)
 +|\widehat{G}[u_{xx}](0)|\|k'\|_{L^2(0,\tau)}
 \nonumber\\ & 
\quad +\frac{1}{|\psi(\ell)|}(|k(0)|+\tau^{1/2}M) \Big(\|f'\|_{L^2(0,\tau)}
 +\|\psi\|_{L^2(I)}(\tau^2 M+\tau^{3/2}\|v_1\|_{L^2(I)}+\tau^{1/2}\|v_0\|_{L^2(I)})\Big).
\end{align}
Then we will use \eqref{eqn-4.8} for $\|k'\|_{L^2(0,\tau)}$ in \eqref{eqn-4.9}. For simplicity, hereafter $C$ represents a generic constant that depends only on $I.$

Using the energy estimate for the linear problem \eqref{eqn-4.2} (see Appendix \ref{linear-EE}), we have
\begin{align*}
    \|v\|_{H^2(0,\tau;H_0^1(I)\cap H^2(I))}\leq C(\|v_0\|_{H_0^1(I)\cap H^2(I)}+\|v_1\|_{H_0^1(I)\cap H^2(I))}+\|K\|_{L^2(0,\tau;L^2(I))}),
\end{align*}
where
\begin{align*}
    K:=- \tilde k(t)u_0''(x)-\int_0^t \tilde k(t-s)\tilde v_{xx}(x,s)\,ds+z''\frac x\ell.
\end{align*}
Next, we estimate $K$ as follows
\begin{align*}
    \|K\|_{L^2(0,\tau;L^2(I))}&\leq \|\tilde k\|_{L^2(0,\tau)}\|u''_0\|_{L^2(I)}
+\tau^{1/2}\|\tilde k\|_{L^2(0,\tau)}\|\tilde v_{xx}\|_{L^2(0,\tau;L^2(I))}+\|z''\|_{L^2(0,\tau)}
\\
& \leq  C\Big((\|u_0\|_{H^2(I)}+\tau^{1/2}\|\tilde v\|_{L^2(0,\tau;H^2(I))})\|\tilde k\|_{L^2(0,\tau)}+\|z''\|_{L^2(0,\tau)}\Big).
\end{align*}
Thus, making use of definition of the space $Q(\tau, M)$ and using $y''(t)=y''(0)+\int_0^ty'''(s)ds$, we obtain the following estimate
\begin{align*}
    \|K\|_{L^2(0,\tau;L^2(I))} & \leq C\Bigg\{\left(\|u_0\|_{H^2(I)}+\tau^{1/2}M\right)\left\|k(0)+\int_0^t\tilde k'(s)ds\right\|_{L^2(0,\tau)}
+p\|y'''\|_{L^2(0,\tau)}
\\
 & \qquad \quad +q\left\|y''(0)+\int_0^ty'''(s)ds\right\|_{L^2(0,\tau)}\Bigg\}
 \\ & \leq  C\Bigg\{\left(\|u_0\|_{H^2(I)}+\tau^{1/2}M\right)\left(\tau^{1/2}|k(0)|+\tau \|\tilde k'\|_{L^2(0,\tau)}\right)
\\
& \qquad \quad +p\|y'''\|_{L^2(0,\tau)}+q\tau^{1/2}|y''(0)|+q\tau \|y'''\|_{L^2(0,\tau)}\Bigg\}.
\end{align*}
Using \eqref{eqn-4.8} and \eqref{eqn-4.9}, we have
\begin{align}\label{eqn-4.10}
    & \|K\|_{L^2(0,\tau;L^2(I))} 
    \nonumber\\ & \leq  C\bigg[\left(\|u_0\|_{H^2(I)}+\tau^{1/2}M\right)\left(\tau^{1/2}|k(0)|+\tau M\right)
  +q\tau^{1/2}|y''(0)| +(p+\tau q)\|y'''\|_{L^2(0,\tau)} \bigg].
\end{align}
If we take $\tau>0$ small enough such that
\begin{align}\label{eqn-4.11}
    \tau^2(1+M)\leq 1,
\end{align}
estimates \eqref{eqn-4.8}-\eqref{eqn-4.10}, transform to
\begin{align*}
  &  \|v\|_{H^2(0,\tau;H_0^1(I)\cap H^2(I))}+\| k'\|_{L^2(0,\tau)}
    \\ & \leq
C\Biggr\{\|v_0\|_{H_0^1(I)\cap H^2(I)}+\|v_1\|_{H_0^1(I)\cap H^2(I)}+\left(\|u_0\|_{H^2(I)}+\tau^{1/2}M\right)\left(\tau^{1/2}|k(0)|+\tau M\right)
\\ & \qquad \quad
+q\tau^{1/2}|y''(0)| +(p+\tau q)\|y'''\|_{L^2(0,\tau)}+\| k'\|_{L^2(0,\tau)}
\Biggr\}
\\ & 
\leq C\Biggr\{\|v_0\|_{H_0^1(I)\cap H^2(I)}+\|v_1\|_{H_0^1(I)\cap H^2(I)}+\left(\|u_0\|_{H^2(I)}+1\right)\left(|k(0)|+1\right)
\\ & \qquad \quad
+q|y''(0)| +(p+q)\Bigg(\frac{1}{|\psi(\ell)|}\Big(\|f''\|_{L^2(0,\tau)}|
+\|\psi\|_{L^2(I)}(1+\|v_1\|_{L^2(I)})\Big)
\\ & \qquad \quad
+\frac{1}{|\psi(\ell)|}(|k(0)|+1) \Big(\|f'\|_{L^2(0,\tau)}
+\|\psi\|_{L^2(I)}(1+\|v_1\|_{L^2(I)}+\|v_0\|_{L^2(I)})\Big)
\\ & \qquad \quad +|\widehat{G}[u_{xx}(0)]| F(\tau)\Bigg)+F(\tau)\Biggr\}=:M_0.
\end{align*}
Here
\begin{align*}
    & F(\tau)
    \nonumber\\ & =\alpha \Biggr[\|f^{(iv)}\|_{L^2(0,\tau)}
+\|\varphi'''\|_{L^2(I)}\Big[1+\|v_1\|_{L^2(I)}+(1+\|v_1\|_{L^2(I)}+\|v_0\|_{L^2(I)})(|k(0)|+1)\Big]\Biggr].
\end{align*}
We notice that $M_0=M_0(\tau)$ is non-decreasing. Therefore, we can define the  mapping $A:Q(\tau, M)\rightarrow  Q(\tau, M)$  by choosing $M\geq M_0$ and $\tau$ small enough as in \eqref{eqn-4.11}.
\vskip 2mm
\noindent
{\bf Step 2.} In this step, we prove that $A$ defined by \eqref{eqn-4.1} is a contraction mapping.

Let $(\tilde v_j,\tilde k'_j)\in Q(\tau,M)$ for $j=1,2$; define $k'_j$ and $v_j$ by \eqref{eqn-4.1}-\eqref{eqn-4.3} with $(\tilde v,\tilde k')=(\tilde v_j,\tilde k'_j)$ respectively. Then $(\tilde v_1,\tilde k'_j)$ for $j=1,2$ satisfies
\begin{equation}\label{eqn-4.12}
    \left\{
    \begin{aligned}
        & (v_1-v_2)_{tt}-(v_1-v_2)_{xx}-\beta (v_1-v_2)_{xxtt}+(\tilde k_1-\tilde k_2)u''_0(x)
\\ &
+\int_0^t (\tilde k_1-\tilde k_2)(t-s) (\tilde v_2)_{xx}(x,s)\,ds+\int_0^t \tilde k_1(t-s)(\tilde v_1-\tilde v_2)_{xx}(x,s)\,ds
\\ & 
=(z''_1-z''_2)\frac x\ell,\quad \ && \hspace{-20mm}  (x,t)\in I\times (0,T),
\\ & 
v_1-v_2\big|_{x=0}=v_1-v_2\big|_{x=\ell}=0, \quad && \hspace{-20mm} 0<t<T,
\\ & 
v_1-v_2\big|_{t=0}=0, \ (v_1-v_2)_t\big|_{t=0}=0, \quad  && \hspace{-20mm} x\in I,
    \end{aligned}
    \right.
\end{equation}
with
\begin{align}\label{eqn-4.13}
    & (k'_1-k'_2)(t)
    \nonumber\\ & =-\alpha\Bigg\{\int_0^\ell  (\tilde v_1-\tilde v_2)_t\varphi'''(x)dx
+k(0)\int_0^\ell (\tilde v_1-\tilde v_2)\varphi'''(x)dx
\nonumber \\ & \quad 
+\int_0^t\int_0^\ell  (\tilde k_1-\tilde k_2)'(t-s) \tilde v_2(x,s)\varphi'''dxds+\int_0^t\int_0^\ell  \tilde k_1'(t-s) (\tilde v_1-\tilde v_2)(x,s)\varphi'''dxds\Bigg\},
\end{align}
and
\begin{align}\label{eqn-4.14}
    (y'''_1-y'''_2)(t)& =\frac{1}{\psi(\ell)}\int_0^\ell \psi(x)(\tilde v_1-\tilde v_2)_{txx}dx-(k_1'-k_2')(t)\widehat{G}[u_{xx}](0)
\nonumber \\ & \quad
-\frac{k(0)}{\psi(\ell)}\int_0^\ell \psi(x)(\tilde v_1-\tilde v_2)_{xx}dx
-\int_0^t(\tilde k_1'-\tilde k_2')(t-s)G[(\tilde v_2)_{xx}](s)\,ds
\nonumber \\ & \quad
-\frac{1}{\psi(\ell)}\int_0^t\tilde k_1'(t-s)\int_0^\ell \psi(x)(\tilde v_1-\tilde v_2)_{xx}(x,s)\,dxds, \quad  0< t<T.
\end{align}
Now we estimate $k'_1-k'_2$ from the equation \eqref{eqn-4.13} as follows:
\begin{align*}
     \|k_1'-k'_2\|_{L^2(0,\tau)}
      & \leq \alpha \Bigg\{\|(\tilde v_1-\tilde v_2)_t\|_{L^2(0,\tau;L^2(I))}\|\varphi'''\|_{L^2(I)}
+|k(0)|\|\tilde v_1-\tilde v_2\|_{L^2(0,\tau;L^2(I))}\|\varphi'''\|_{L^2(I)}
\nonumber\\ & \quad \qquad +\left\|\int_0^t\int_0^\ell (\tilde k_1'-\tilde k'_2)(t-s) \tilde v_2(x,s)\varphi'''(x)\,dxds\right\|_{L^2(0,\tau)}
\\ & \quad \qquad 
+\left\|\int_0^t\int_0^\ell \tilde k'_1(t-s)(\tilde v_1-\tilde v_2)(x,s)\varphi'''(x)\,dxds\right\|_{L^2(0,\tau)}\Bigg\}.
\end{align*}
We use the same argument as in the calculation of \eqref{eqn-4.8} to obtain 
\begin{align}\label{eqn-4.15}
   & \|k_1'-k'_2\|_{L^2(0,\tau)}
   \nonumber\\ &  \leq \alpha \Bigg\{\tau\|(\tilde v_1-\tilde v_2)_{tt}\|_{L^2(0,\tau;L^2(I))}\|\varphi'''\|_{L^2(I)}
+\tau^2|k(0)|\|(\tilde v_1-\tilde v_2)_{tt}\|_{L^2(0,\tau;L^2(I))}\|\varphi'''\|_{L^2(I)}
\nonumber \\ & \qquad 
+\tau^{1/2}\|\tilde k_1'-\tilde k'_2\|_{L^2(0,\tau)}\|\varphi'''\|_{L^2(I)}\|\tilde v_2\|_{L^2(0,\tau;L^2(I))}
\nonumber \\ & \qquad 
+\tau^{1/2}\|\tilde k_1'\|_{L^2(0,\tau)}\|\varphi'''\|_{L^2(I)}\|\tilde v_1-\tilde v_2\|_{L^2(0,\tau;L^2(I))}\Bigg\}
\nonumber \\ &
\leq \alpha\tau^{1/2} \Bigg\{\tau^{1/2}\|(\tilde v_1-\tilde v_2)_{tt}\|_{L^2(0,\tau;L^2(I))}\|\varphi'''\|_{L^2(I)}
+\tau^2|k(0)|\|(\tilde v_1-\tilde v_2)_{tt}\|_{L^2(0,\tau;L^2(I))}\|\varphi'''\|_{L^2(I)}
\nonumber \\ & \qquad +\tau^{1/2}\|\tilde k_1'-\tilde k'_2\|_{L^2(0,\tau)}\|\varphi'''\|_{L^2(I)}\|\tilde v_2\|_{L^2(0,\tau;L^2(I))}
\nonumber \\ &
 \qquad +\tau^{1/2}\|\tilde k_1'\|_{L^2(0,\tau)}\|\varphi'''\|_{L^2(I)}\|\tilde v_1-\tilde v_2\|_{L^2(0,\tau;L^2(I))}\Bigg\}
\nonumber \\ &
\leq \alpha\tau^{1/2}\|\varphi'''\|_{L^2(I)} \left(\tau^{1/2}+\tau^{3/2}|k(0)|+\tau^{2}M\right)\|\tilde v_1-\tilde v_2\|_{H^2(0,\tau;H_0^1(I)\cap H^2(I))}
\nonumber \\ & \quad
+\alpha\tau^{1/2}\|\varphi'''\|_{L^2(I)}M\|\tilde k_1'-\tilde k'_2\|_{L^2(0,\tau)}.
\end{align}
The term \eqref{eqn-4.14} can be estimated as 
\begin{align}\label{eqn-4.16}
  &  \|y'''_1-y'''_2\|_{L^2(0,\tau)}
  \nonumber\\ & \leq \frac{1}{|\psi(\ell)|}\|\psi\|_{L^2(I)}\tau\|(\tilde v_1-\tilde v_2)_{tt}\|_{L^2(0,\tau;L^2(I))}+|\widehat{G}[u_{xx}](0)|\|k_1'-k'_2\|_{L^2(0,\tau)}
  \nonumber \\ & \quad
 +\frac{|k(0)|}{|\psi(\ell)|}\|\psi\|_{L^2(I)}\tau^2\|(\tilde v_1-\tilde v_2)_{tt}\|_{L^2(0,\tau;L^2(I))}
+\tau^{1/2}\|\tilde k_1'-\tilde k'_2\|_{L^2(0,\tau)}\|G[(\tilde v_2)_{xx}]\|_{L^2(0,\tau;L^2(I))}
\nonumber \\ & \quad  +
\frac{\tau^{1/2}}{|\psi(\ell)|}\|\psi\|_{L^2(I)}\|\tilde k_1'\|_{L^2(0,\tau)}\|(\tilde v_1-\tilde v_2)_{xx}\|_{L^2(0,\tau;L^2(I))}
\nonumber \\ & 
\leq \frac{\tau^{1/2}}{|\psi(\ell)|}\|\psi\|_{L^2(I)}\left(\tau^{1/2}+|k(0)|\tau^{3/2}+M\right)\|\tilde v_1-\tilde v_2\|_{H^2(0,\tau;H_0^1(I)\cap H^2(I))}
\nonumber \\ & \quad +\frac{\tau^{1/2}}{|\psi(\ell)|}\left(\|f'\|_{L^2(0,\tau)}+M\|\psi\|_{L^2(I)}\right)\|\tilde k_1'-\tilde k'_2\|_{L^2(0,\tau)} +|\widehat{G}[u_{xx}](0)|\|k_1'-k'_2\|_{L^2(0,\tau)}.
\end{align}
Using the estimates for the linear problem \eqref{eqn-4.12} (see Appendix \ref{linear-EE}), we estimate
\begin{align*}
    \|v_1-v_2\|_{H^2(0,\tau;H_0^1(I)\cap H^2(I))}\leq C\|K_1\|_{L^2(0,\tau;L^2(I))},
\end{align*}
where
\begin{align*}
    K_1& :=-(\tilde k_1-\tilde k_2)u''_0(x)-\int_0^t (\tilde k_1-\tilde k_2)(t-s) (\tilde v_2)_{xx}(x,s)\,ds
\nonumber \\ & \quad -\int_0^t \tilde k_1(t-s)(\tilde v_1-\tilde v_2)_{xx}(x,s)\,ds+(z''_1-z''_2)\frac x\ell.
\end{align*}
For $K_1$, we obtain
\begin{align}\label{eqn-4.17}
    \|K_1\|_{L^2(0,\tau;L^2(I))} & \leq \|\tilde k_1-\tilde k_2\|_{L^2(0,\tau)}\|u_0\|_{H^2(I)}
+\tau^{1/2}\|\tilde k_1-\tilde k_2\|_{L^2(0,\tau)}\|(\tilde v_2)_{xx}\|_{L^2(0,\tau;L^2(I))}
\nonumber\\ & \quad +\tau^{1/2}\|\tilde k_1\|_{L^2(0,\tau)}\|(\tilde v_1-\tilde v_2)_{xx}\|_{L^2(0,\tau;L^2(I))}
+\|z''_1-z_2''\|_{L^2(0,\tau)}
\nonumber\\ & 
\leq C\Big(\tau\|\tilde k'_1-\tilde k'_2\|_{L^2(0,\tau)}\|u_0\|_{H^2(I)}
+\tau^{3/2}\|\tilde k'_1-\tilde k'_2\|_{L^2(0,\tau)}\|\tilde v_2\|_{L^2(0,\tau;H^2(I))}
\nonumber\\ & \quad +\tau^{3/2}\|\tilde k'_1\|_{L^2(0,\tau)}\|\tilde v_1-\tilde v_2\|_{L^2(0,\tau;H_0^1(I)\cap H^2(I))}
+\|z''_1-z_2''\|_{L^2(0,\tau)}\Big)
\nonumber\\ & \leq C\Big(\tau \left(\|u_0\|_{H^2(I)}+\tau^{1/2}M\right)\|\tilde k_1'-\tilde k_2'\|_{L^2(0,\tau)}
\nonumber\\ & \quad +\tau^{3/2}M\|\tilde v_1-\tilde v_2\|_{H^2(0,\tau;H_0^1(I)\cap H^2(I))}+(p+q\tau)\|y_1'''-y'''_2\|_{L^2(0,\tau)}\Big).
\end{align}
Thus, combining \eqref{eqn-4.15}-\eqref{eqn-4.17}, we arrive at
\begin{align*}
& \|v_1-v_2\|_{H^2(0,\tau;H_0^1(I)\cap H^2(I))}+\|k_1'-k'_2\|_{L^2(0,\tau)}
\nonumber \\ &
\leq C\Big(\tau \left(\|u_0\|_{H^2(I)}+\tau^{1/2}M\right)\|\tilde k_1'-\tilde k_2'\|_{L^2(0,\tau)}
\nonumber\\ & \quad +\tau^{3/2}M\|\tilde v_1-\tilde v_2)\|_{H^2(0,\tau;H_0^1(I)\cap H^2(I))}+(p+q\tau)\|y_1'''-y'''_2\|_{L^2(0,\tau)}\Big)+\|k_1'-k'_2\|_{L^2(0,\tau)}
\nonumber\\ & \leq  C\Big(\tau \left(\|u_0\|_{H^2(I)}+\tau^{1/2}M\right)\|\tilde k_1'-\tilde k_2'\|_{L^2(0,\tau)}
+\tau^{3/2}M\|\tilde v_1-\tilde v_2\|_{H^2(0,\tau;H_0^1(I)\cap H^2(I))}
\nonumber\\ & \quad +(p+q\tau)\Big\{
\frac{\tau^{1/2}}{|\psi(\ell)|}\|\psi\|_{L^2(I)}\left(\tau^{1/2}+|k(0)|\tau^{3/2}+M\right)\|\tilde v_1-\tilde v_2\|_{H^2(0,\tau;H_0^1(I)\cap H^2(I))}
\nonumber\\ & \quad +\frac{\tau^{1/2}}{|\psi(\ell)|}\left(\|f'\|_{L^2(0,\tau)}+M\|\psi\|_{L^2(I)}\right)\|\tilde k_1'-\tilde k'_2\|_{L^2(0,\tau)}\Big\} \Big)
\nonumber\\ & \quad +\left(1+C(p+q\tau)|\widehat{G}[u_{xx}](0)\right)|\|k_1'-k'_2\|_{L^2(0,\tau)}.
\end{align*}
Using \eqref{eqn-4.15}, finally we arrive at
\begin{align*}
& \|v_1-v_2\|_{H^2(0,\tau;H_0^1(I)\cap H^2(I))}+\|k_1'-k'_2\|_{L^2(0,\tau)}
\nonumber\\ & \leq \tau^{1/2}\Biggl\{\tau^{1/2}C\left(\|u_0\|_{H^2(I)}+\tau^{1/2}M\right)+\frac{p+q\tau}{\psi(\ell)}
\left(\|f'\|_{L^2(0,\tau)}+M\|\psi\|_{L^2(I)}\right)
\nonumber\\ & \quad 
+\alpha M\left(1+C(p+q\tau)|\widehat{G}[u_{xx}](0)\right)\Biggl\}\|\tilde k_1'-\tilde k'_2\|_{L^2(0,\tau)}
\nonumber\\ & \quad
+\tau^{1/2}\Biggl\{C\tau M+C\frac{p+q\tau}{\psi(\ell)}\|\psi\|_{L^2(I)}\left(\tau^{1/2}+|k(0)|\tau^{3/2}+M\right)
\nonumber\\ & \quad
+\alpha \|\varphi'''\|_{L^2(I)}\left(1+C(p+q\tau)|\widehat{G}[u_{xx}](0)\right)\left(\tau^{1/2}+|k(0)|\tau^{3/2}+\tau^2M\right)\Biggr\}
\nonumber\\ & \quad 
\times\|\tilde v_1-\tilde v_2\|_{H^2(0,\tau;H_0^1(I)\cap H^2(I))}.    
\end{align*}
If we take $\tau>0$ small enough, then we have $A: Q(\tau, M) \rightarrow  Q(\tau, M)$ is a \emph{contraction mapping}. Thus, from the \emph{contraction mapping theorem}, one can conclude that for a small time $\tau,$ there exists a unique solution $(v,k')\in H^2(0,\tau;H_0^1(I)\cap H^2(I))\times L^2(0,\tau)$ to the equivalent  system \eqref{eqn-3.2}--\eqref{eqn-3.6} in $[0,\tau]$. The proof is completed.
\end{proof}

Next, we prove the global uniqueness result for our inverse problem. Let $0<\tau <  \tau'\leq (2\tau)\land T$ and let $(\hat v, \hat k)$ be the solution of the system \eqref{eqn-3.2}--\eqref{eqn-3.6} obtained in $[0,\tau]$ from Theorem \ref{thm-4.1} for small $\tau>0.$ We will prove that $(\hat v, \hat k)$ is extensible to a solution $(v, k)\in H^2(0,\tau';H_0^1(I)\cap H^2(I))\times H^1(0,\tau')$ in $[0,\tau'].$ The following lemma is very useful in the proof of the Theorem \ref{thm-4.3} which gives the global existence of the solution of the system \eqref{eqn-3.2}--\eqref{eqn-3.6}.

\begin{lemma}\label{lem-4.1}
Let the assumptions $(i1)$-$(i5)$ hold. Let $0<\tau <  \tau'\leq (2\tau)\land T$ and  let $(u,k)$ (or $(v,k)$) be a solution of the system \eqref{eqn-3.1}--\eqref{eqn-3.6} in $(0,\tau')$ with the regularity
$$
v\in H^2(0,\tau';H_0^1(I)\cap H^2(I)), \quad k\in  H^1(0,\tau').
$$
We set
\begin{itemize}
    \item  [(H1)] $ v:=u_t+z\frac{x}{\ell},$
\item  [(H2)] $ \hat v:=v|_{[0,\tau]},$
\item  [(H3)] $  v_{\tau}:=v(x,\tau+t),\quad t\in [0,\tau'-\tau],$
\item  [(H4)] $ \hat k:=k|_{[0,\tau]},$
\item  [(H5)] $  k_{\tau}:=k(\tau+t),\quad t\in [0,\tau'-\tau],$
\item  [(H6)] $ g(x,t):=\int_t^{\tau}\hat k(\tau+t-s)\hat v_{xx}ds,\quad t\in [0,\tau'-\tau],$
\item  [(H7)] $ u_{\tau}(x):=v(x,\tau), \quad u_{\tau \tau}(x):=v_t(x,\tau),$
\item [(H8)] $ f_{\tau}:=f(\tau+t),\quad t\in [0,\tau'-\tau],$
\item [(H9)] $ y_{\tau}:=y(\tau+t),\quad t\in [0,\tau'-\tau].$
\end{itemize}
\medskip

Then we have the following obvious conditions:
\begin{itemize}
    \item [(A1)]$\ v_{\tau}\in H^2(0,\tau'-\tau;H_0^1(I)\cap H^2(I)),$
\item [(A2)] $\ k_{\tau}\in H^1(0,\tau'-\tau),$
\item  [(A3)]  $\  u_{\tau},u_{\tau \tau}\in  H_0^1(I)\cap H^2(I),$

\item [(A4)] $\ f_{\tau} \in H^4(0,\tau'-\tau),$

\item [(A5)]  $\ g(x,t) \in H^1(0,\tau'-\tau; L^2(I)),$

\item[(A6)] $\ (v_{\tau},k_{\tau}) $ solves the system:
\begin{align}\label{eqn-4.18}
   & (v_{\tau})_{tt}-(v_{\tau})_{xx}-\beta (v_{\tau})_{xxtt}+k_{\tau}u''_0(x)+\int_0^t k_{\tau}(t-s)\hat v_{xx}
(x,s)\,ds
\nonumber\\ & +\int_0^t \hat k(t-s) (v_{\tau})_{xx}(x,s)\,ds
+g(x,t)=z_{\tau}''\frac x\ell,\ (x,t)\in I\times (0,\tau'-\tau),
\end{align}
\begin{align}\label{eqn-4.19}
    v\big|_{x=0}=v\big|_{x=\ell}=0, \quad 0<t<\tau'-\tau, 
\end{align}
\begin{align}\label{eqn-4.20}
    v_{\tau}\big|_{t=0}=u_{\tau}(x), \ (v_{\tau})_t\big|_{t=0}=u_{\tau \tau}(x), \quad  x\in I,
\end{align}
with
\begin{align}\label{eqn-4.21}
    k_{\tau}'(t)&=\alpha\Bigg\{f_{\tau}^{(iv)}(t)+\int_0^\ell (v_{\tau})_t(x,t)\varphi'''(x)dx
-k(0)\int_0^\ell  v_{\tau}(x,t)\varphi'''(x)dx
\nonumber\\ & \quad -\int_0^t\int_0^\ell k_{\tau}'(t-s) \hat v(x,s)\varphi'''(x)dxds
-\int_0^t\int_0^\ell \hat k'(t-s) (v_{\tau})(x,s)\varphi'''(x)dxds
 \nonumber\\ & \quad   -\int_t^{\tau}\int_0^\ell \hat k'(\tau+t-s)\hat v(x,s)\varphi'''(x)dxds\Bigg\},
\end{align}
where $z_{\tau}'':=py_{\tau}'''+qy_{\tau}''$,
and
\begin{align}\label{eqn-4.22}
    y_{\tau}'''(t)& :=G'[(v_{\tau})_{xx}](t)-k_{\tau}'(t)\widehat{G}[u_{xx}](0)
-\hat k(0)G[(v_{\tau})_{xx}](t)
\nonumber \\ & \quad 
-\int_0^t\hat k'(t-s)G[(v_{\tau})_{xx}](s)\,ds-\int_0^tk_{\tau}'(t-s)G[\hat v_{xx}](s)\,ds
\nonumber \\ & \quad
-\int_t^{\tau}\hat k'(\tau+t-s)G[\hat v_{xx}](s)\,ds, \quad  0< t<\tau'-\tau,
\end{align}
where $G[(v_{\tau})_{xx}](t)=\frac{1}{\psi(\ell)}\left(f'_{\tau}(t)+\int_0^\ell \psi(x)[(v_{\tau})_{xx}](x,t)dx\right).$

\item[(A7)] On the over hand, if $(\hat v,\hat k) \in H^2(0,\tau;H_0^1(I)\cap H^2(I)) \times H^1(0,\tau)$, solves the system \eqref{eqn-3.2}--\eqref{eqn-3.6} and  $(v_{\tau},k_{\tau})\in H^2(0,\tau'-\tau;H_0^1(I)\cap H^2(I))\times H^1(0,\tau'-\tau)$ solves the system \eqref{eqn-4.18}--\eqref{eqn-4.22}. Set for $t\in [0,\tau'],$
\begin{align}\label{eqn-4.23}
    (v(x,t),k(t))=\begin{cases}
(\hat v,\hat k),&\text{if $t \in [0,\tau]$,}\\
(v_{\tau},k_{\tau}),&\text{if $t \in [\tau,\tau']$.}
\end{cases}
\end{align}
\end{itemize}
Then $(v,k)  \in H^2(0,\tau';H_0^1(I)\cap H^2(I)) \times H^1(0,\tau')$ and solves the system \eqref{eqn-3.2}--\eqref{eqn-3.6} in $[0,\tau'].$
\end{lemma}
\begin{proof}
 The assertions (A1)-(A5) and (A7) can be easily verified. Let us now prove (A6). First of all, from the system \eqref{eqn-3.2}--\eqref{eqn-3.6}, we have
\begin{align}\label{eqn-4.24}
  &  v_{tt}(x,\tau+t)-v_{xx}(x,\tau+t)-\beta  v_{xxtt}(x,\tau+t)+k(\tau+t)u''_0(x)
\nonumber \\
& +\int_0^{\tau+t} k(\tau+t-s)v_{xx}(x,s)\,ds=z''(\tau+t)\frac x\ell,\ (x,t)\in I\times (0,\tau'-\tau),
\end{align}
\begin{align}\label{eqn-4.25}
    v(x,\tau+t)\big|_{x=0}=v(x,\tau+t)\big|_{x=\ell}=0, \quad 0<t<\tau'-\tau, 
\end{align}
\begin{align}\label{eqn-4.26}
    v(x,\tau)=u_{\tau}(x), \ v_t(x,\tau)=u_{\tau \tau}(x), \quad  x\in I,
\end{align}
and from \eqref{eqn-3.7} and \eqref{eqn-2.6} for all $t\in [0,\tau'-\tau]$
\begin{align}\label{eqn-4.27}
    k(\tau+t) &=\alpha\Bigg\{f'''(\tau+t)+\int_0^\ell v(x,\tau+t)\varphi'''(x)dx
\nonumber \\ & \quad
-\int_0^{\tau+t}\int_0^\ell k(\tau+t-s) v(x,s)\varphi'''(x)dxds\Bigg\},
\end{align}
\begin{align}\label{eqn-4.28}
     y''(\tau+t)=G[v_{xx}](\tau+t)-k(\tau+t)\widehat{G}[u_{xx}](0)-\int_0^{\tau+t}k(\tau+t-s)G[v_{xx}](s)ds.
\end{align}
Next, as $\tau'-\tau\leq \tau,$ if $t\in [0,\tau'-\tau],$ we have 
\begin{align}\label{eqn-4.29}
  &  \int_0^{\tau+t}k(\tau+t-s)v_{xx}ds
\nonumber \\ &
= \int_0^{t}k(\tau+t-s)v_{xx}ds+\int_t^{\tau}k(\tau+t-s)v_{x x}ds+\int_{\tau}^{\tau+t}k(\tau+t-s)v_{xx}ds
\nonumber \\ & 
= \int_0^{t}k_{\tau}(t-s)\hat v_{xx}ds+\int_t^{\tau}\hat k(\tau+t-s)\hat v_{xx}ds+\int_{0}^{t}\hat k(t-s)(v_{\tau})_{xx}ds.
\end{align}
Substituting \eqref{eqn-4.29} in \eqref{eqn-4.24}--\eqref{eqn-4.26}, we immediately have \eqref{eqn-4.18}--\eqref{eqn-4.22}. Then, substituting \eqref{eqn-4.29} in \eqref{eqn-3.7} and \eqref{eqn-2.6} and differentiating we get \eqref{eqn-4.21}, \eqref{eqn-4.22}. The proof of the lemma is completed. 
\end{proof}

\begin{theorem}[Global-in-time uniqueness]\label{thm-4.2}
Let the assumptions $(i1)$-$(i5)$ hold. Then, if $\tau \in (0,T],$ and the inverse problem has two solutions
\begin{align*}
    v_j\in H^2(0,\tau;H_0^1(I)\cap H^2(I)),\  k_j\in H^1(0,\tau), y_j\in H^3(0,\tau) ,\ j\in{1,2},
\end{align*}
then $(v_1,k_1,y_1)=(v_2,k_2,y_2).$
\end{theorem}

\begin{proof} 
It suffices to show that $v_1=v_2$ and $k'_1=k_2'.$ We need to prove
\begin{align}\label{eqn-4.30}
  \|v_1-v_2\|_{H^2(0,\tau;H_0^1(I)\cap H^2(I))}+\|k'_1-k'_2\|_{L^2(0,\tau)}=0.
\end{align}
We have proved in Theorem \ref{thm-4.1} that for every $M\in \mathbb{R_+},$ there exists $\tau(M)\in (0,T]$, such that for all $t \in (0,\tau(M)],$ system \eqref{eqn-3.2}--\eqref{eqn-3.6} has a unique solution
\begin{align*}
    (\tilde v,\tilde k)\in H^2(0,\tau(M);H_0^1(I)\cap H^2(I))\times H^1(0,\tau(M)),
\end{align*}
such that
\begin{align*}
    \|\tilde v\|_{H^2(0,\tau(\ell);H_0^1(I)\cap H^2(I))}+\|\tilde k'\|_{L^2(0,\tau(M))}\leq M.
\end{align*}
This implies that, if $(v_j,k_j)\in H^2(0,\tau;H_0^1(I)\cap L^2(I)) \times H^1(0,\tau),\ j\in\{1,2\},$ are solutions of the system \eqref{eqn-3.2}--\eqref{eqn-3.6}, then there exists $\tau_1 \in (0,\tau]$, such that $(v_1,k_1)$ and $(v_2,k_2)$ coincide on $[0,\tau_1].$ In fact, we can set
\begin{align*}
    M_1:=\max\Big\{\|\tilde v_j\|_{H^2(0,\tau;H_0^1(I)\cap H^2(I))}+\|\tilde k_j'\|_{L^2(0,\tau)}:j\in \{1,2\}\Big\}
\end{align*}
and obtain that $(v_1,k_1)$ and $(v_2,k_2)$ coincide on $[0,\tau(M_1)]\wedge \tau ]$. Let us define
\begin{align}\label{eqn-4.31}
    \tau_1:=\inf\Big\{t\in (0,\tau]:\|v_1-v_2\|_{H^2(0,\tau;H_0^1(I)\cap H^2(I))}+\|k'_1-k'_2\|_{L^2(0,\tau)}>0\Big\}.
\end{align} 
If \eqref{eqn-4.30} is not true, then it is obvious that $\tau_1$ is well defined and we have $\tau_1\in (0,\tau).$ Set $t\in [0,\tau-\tau_1],\ j\in  \{1,2\},$
\begin{align*}
    v^*_j(t)=v_j(\tau_1+t),\quad k^*_j(t)=k_j(\tau_1+t), \quad y^*_j(t)=y_j(\tau_1+t).
\end{align*}
Keeping in mind that $v_1=v_2$ and $k_1=k_2$  if $t \in [0,\tau_1]$. From conditions (A1)-(A7) in Lemma \ref{lem-4.1}, we obtain
\begin{align}\label{eqn-4.32}
   & (v_1^*-v_2^*)_{tt}-(v_1^*-v_2^*)_{xx}-\beta (v_1^*-v_2^*)_{xxtt}+(k_1^*-k_2^*)u''_0(x)
\nonumber \\ & 
\quad+\int_0^t (k_1^*-k_2^*)(t-s)(v_1)_{xx}(x,s)\,ds+\int_0^t k_1(t-s)(v_1^*-v_2^*)_{xx}(x,s)\,ds
\nonumber \\ & 
=(z^*_1-z_2^*)''(t)\frac x\ell,\ (x,t)\in I\times [0,\tau_1\wedge (\tau-\tau_1)],
\end{align}
\begin{align}\label{eqn-4.33}
    (v_1^*-v_2^*)\big|_{x=0}=(v_1^*-v_2^*)\big|_{x=\ell}=0, \quad 0\leq t\leq \tau_1\wedge (\tau-\tau_1),
\end{align}
\begin{align}\label{eqn-4.34}
    (v_1^*-v_2^*)\big|_{t=0}=0, \ (v_1^*-v_2^*)_t\big|_{t=0}=0, \quad  x\in I, 
\end{align}
\begin{align}\label{eqn-4.35}
   & (k_1^*-k_2^*)'(t)
   \nonumber \\ & =\alpha\Bigg\{\int_0^\ell (v_1^*-v_2^*)_t(x,t)\varphi'''(x)dx
- k(0)\int_0^\ell  (v_1^*-v_2^*)(x,t)\varphi'''(x)dx
\nonumber \\ &\quad  -\int_0^t\int_0^\ell (k_1^*-k_2^*)'(t-s) v_1(x,s)\varphi'''(x)dxds 
-\int_0^t\int_0^\ell k_1'(t-s) (v_1^*-v_2^*)(x,s)\varphi'''(x)dxds\Bigg\},
\end{align}
where $(z^*_1-z^*_2)'':=p(y^*_1-y^*_2)'''+q(y^*_1-y^*_2)''$,
\begin{align}\label{eqn-4.36}
    (y^*_1-y^*_2)'''(t)&:=G'[(v_1^*-v_2^*)_{xx}](t)-(k_1^*-k_2^*)'(t)\widehat{G}[u_{xx}](0)
-k(0)G[(v_1^*-v_2^*)_{xx}](t)
\nonumber \\ & \quad
-\int_0^t k_1'(t-s)G[(v_1^*-v_2^*)_{xx}](s)\,ds-\int_0^t(k_1^*-k_2^*)'(t-s)G[(v_1)_{xx}](s)\,ds
\end{align}
and
\begin{align}\label{eqn-4.37}
    G[(v_1^*-v_2^*)_{xx}](t)=\frac{1}{\psi(\ell)}\int_0^\ell \psi(x)(v_1^*-v_2^*)_{xx}(x,t)dx=\frac{1}{\psi(\ell)}\int_0^\ell \psi''(x)(v_1^*-v_2^*)(x,t)dx.
\end{align}

Assume that $\delta \in (0,\tau_1\wedge (\tau-\tau_1)].$ Applying similar arguments as in the proof of  Theorem \ref{thm-4.1}, we estimate $(k^*_1-k^*_2)' $ and $(z^*_1-z^*_2)''$ from \eqref{eqn-4.35}, \eqref{eqn-4.36}, using \eqref{eqn-4.37} as
\begin{align*}
    \|(k^*_1-k^*_2)'\|_{L^2(0,\delta)} & \leq \alpha\Big\{\|(v_1^*-v_2^*)_t\|_{L^2(0,\delta;L^2(I))}\|\varphi'''\|_{L^2(I)}
\nonumber \\ & \quad
+|k(0)|\|(v_1^*-v_2^*)\|_{L^2(0,\delta;L^2(I))}\|\varphi'''\|_{L^2(I)}
\nonumber \\ & \quad
+\delta^{1/2}\|(k_1^*-k_2^*)'\|_{L^2(0,\delta)}\|\varphi'''\|_{L^2(I)}\|v_1\|_{L^2(0,\delta;L^2(I))}
\nonumber \\ & \quad
+\delta^{1/2}\|(v_1^*-v_2^*)\|_{L^2(0,\delta;L^2(I))}\|\varphi'''\|_{L^2(I)}\|k'_1\|_{L^2(0,\delta;L^2(I))}\Big\}
\nonumber \\ &  
\leq C(\delta)\|(v_1^*-v_2^*)_t\|_{L^2(0,\delta;L^2(I))}+C\delta^{1/2}\|(k_1^*-k_2^*)'\|_{L^2(0,\delta)}.
\end{align*}
Choosing $\delta>0$ such that $C\delta^{1/2}\leq \frac 12$, we find
\begin{align}\label{eqn-4.38}
    \|(k_1^*-k_2^*)'\|_{L^2(0,\delta)}\leq C(\delta)\|(v_1^*-v_2^*)_t\|_{L^2(0,\delta;L^2(I))}.
\end{align}
Next, we estimate 
\begin{align*}
   & \|(z^*_1-z^*_2)''\|_{L^2(0,\delta)}
   \\ &\leq (p+q\delta)\|(y^*_1-y^*_2)'''\|_{L^2(0,\delta)}
    \\ & \leq (p+q\delta)\Bigl(\frac{1}{|\psi(\ell)|}\|\psi''\|_{L^2(I)}\|(v^*_1-v^*_2)_t\|_{L^2(0,\delta;L^2(I))}+|\widehat{G}[u_{xx}](0)|\|(k^*_1-k^*_2)'\|_{L^2(0,\delta)}
\\ & \quad
+\frac{|k(0)|}{|\psi(\ell)|}\|\psi''\|_{L^2(I)}\|(v^*_1-v^*_2)\|_{L^2(0,\delta;L^2(I))}
+\delta^{1/2}\|(k^*_1-k^*_2)'\|_{L^2(0,\delta)}\|\psi''\|_{L^2(I)}\|v_1\|_{L^2(0,\delta;L^2(I))}
\\ & \quad +
\frac{\delta^{1/2}}{|\psi(\ell)|}\|\psi''\|_{L^2(I)}\|k_1'\|_{L^2(0,\delta)}\|(v^*_1-v^*_2)\|_{L^2(0,\delta;L^2(I))}\Bigr). 
\end{align*}
Using \eqref{eqn-4.38} and Lemma \ref{lem-2.2}, we find
\begin{align}\label{eqn-4.39}
    \|(z^*_1-z^*_2)''\|_{L^2(0,\delta)}\leq C(\delta) \|(v_1^*-v_2^*)_t\|_{L^2(0,\delta;L^2(I))}.
\end{align}

Taking the inner product with $(v_1^*-v_2^*)_t(\cdot)$ to the equation \eqref{eqn-4.32} and then using integration by parts, we obtain 
\begin{align*}
& \frac 12\frac{d}{dt}\left(\|(v_1^*-v_2^*)_t\|^2_{L^2(I)}+\|(v_1^*-v_2^*)_x\|^2_{L^2(I)}+\beta\|(v_1^*-v_2^*)_{xt}\|^2_{L^2(I)}\right)
\\ & 
=((k_1^*-k_2^*)\ast (v_1)_x(t),(v_1^*-v_2^*)_{xt}(t))+((k_1^*-k_2^*) u'_0,(v_1^*-v_2^*)_{xt}(t))
\\ & \quad 
+((k_1\ast (v^*_1-v^*_2)_x(t),(v_1^*-v_2^*)_{xt}(t))-((z^*_1-z_2^*)''(t) \frac{x^2}{2\ell},(v_1^*-v_2^*)_{xt}(t) )
\\ & 
\leq \|(k_1^*-k_2^*)\ast (v_1)_x(t)\|_{L^2(I)}\|(v_1^*-v_2^*)_{xt}(t)\|_{L^2(I)}+|(k_1^*-k_2^*)(t)|\|u'_0\|_{L^2(I)}\|(v_1^*-v_2^*)_{xt}(t)\|_{L^2(I)}
\\ & \quad 
+\|k_1\ast (v_1^*-v^*_2)_x(t)\|_{L^2(I)}\|(v_1^*-v_2^*)_{xt}(t)\|_{L^2(I)}+\frac{\ell}{2}|(z^*_1-z_2^*)''(t)|\|(v_1^*-v_2^*)_{xt}(t)\|_{L^2(I)}
\\ & 
\leq \frac 32\|(v_1^*-v_2^*)_{xt}(t)\|^2_{L^2(I)}+\frac 12\|(k_1^*-k_2^*)\ast (v_1)_x(t)\|^2_{L^2(I)}+\frac 12|(k_1^*-k_2^*)(t)|^2\|u'_0\|^2_{L^2(I)}
\\ & \quad
+\frac 12 \|k_1\ast (v_1^*-v^*_2)_x(t)\|^2_{L^2(I)}+\frac{\ell}{4}|(z^*_1-z_2^*)''(t)|^2+\frac 12\|(v_1^*-v_2^*)_{xt}(t)\|^2_{L^2(I)},
\end{align*}
for a.e. $t\in [0,\delta]$.
Integrating the above inequality from $0$ to $t$, we deduce
\begin{align*}
&  \|(v_1^*-v_2^*)_t\|^2_{L^2(I)}+ \|(v_1^*-v_2^*)_x\|^2_{L^2(I)}+\beta\|(v_1^*-v_2^*)_{xt}\|^2_{L^2(I)}
\\ & 
\leq \|(v_1^*-v^*_2)_x(0)\|^2_{L^2(I)}+\|(v_1^*-v^*_2)_{xt}(0)\|^2_{L^2(I)}+4\int_0^t\|(v_1^*-v_2^*)_{xt}(s)\|^2_{L^2(I)}ds
\\ & \quad 
+t\|k_1^*-k_2^*\|^2_{L^2(0,t)}\int_0^t\|(v_1)_x(s)\|^2_{L^2(I)}ds+\|k_1^*-k_2^*\|^2_{L^2(0,t)}\|u'_0\|^2_{L^2(I)}
\\ & \quad 
+\|(k_1^*-k_2^*)\|^2_{L^2(0,t)}\|u'_0\|^2_{L^2(I)}+t\|k_1\|^2_{L^2(0,t)}\int_0^t\|(v_1^*-v^*_2)_x(s)\|^2_{L^2(I)}ds
+\frac {\ell}{2}\|(z_1^*-z_2^*)\|^2_{L^2(0,t)}
\\ & 
\leq \|(v_1^*-v^*_2)_x(0)\|^2_{L^2(I)}+\|(v_1^*-v^*_2)_{xt}(0)\|^2_{L^2(I)}+C(\delta)\|(v_1^*-v_2^*)_t\|^2_{L^2(0,t)}\|u'_0\|^2_{L^2(I)}
\\ & \quad +C(\delta)\|(v_1^*-v_2^*)_t\|^2_{L^2(0,t)}\int_0^t\|(v_1)_x(s)\|^2_{L^2(I)}ds
\\ & \quad +(4\vee \delta\|k_1\|^2_{L^2(0,t)})\int_0^t\left(\|(v_1^*-v_2^*)_{xt}(s)\|^2_{L^2(I)}+\|(v_1^*-v^*_2)_x(s)\|^2_{L^2(I)}\right)ds
\\ & \quad +\frac l2 \|(v_1^*-v_2^*)_t\|^2_{L^2(0,t;L^2(I))},
\end{align*}
for all $t\in [0,\delta]$, where we have used \eqref{eqn-4.38} and \eqref{eqn-4.39}. An  application of Gronwall's inequality gives
\begin{align*}
   & \sup\limits_{t\in [0,\delta]}\left(\|(v_1^*-v^*_2)_x(t)\|^2_{L^2(I)}+\beta \|(v_1^*-v^*_2)_{xt}(t)\|^2_{L^2(I)}\right)
    \\ & \leq C\left(\|(v_1^*-v^*_2)_x(0)\|_{L^2(I)},\|(v_1^*-v^*_2)_{xt}(0)\|_{L^2(I)}, \delta\right).
\end{align*}
Taking the inner product with $(v_1^*-v_2^*)_{tt}(\cdot)$ to the equation \eqref{eqn-4.32}, we have 
\begin{align*}
& \|(v_1^*-v_2^*)_{tt}(t)\|^2_{L^2(I)}+\beta \|(v_1^*-v_2^*)_{xtt}(t)\|^2_{L^2(I)}
\\ & =-((v_1^*-v_2^*)_{x}(t),(v_1^*-v_2^*)_{xtt})
\\ & \quad 
+((k_1^*-k_2^*)\ast (v_1)_x(t),(v_1^*-v_2^*)_{xtt}(t))-((k_1^*-k_2^*) u''_0,(v_1^*-v_2^*)_{tt}(t))
\\ & \quad
+((k_1\ast (v^*_1-v^*_2)_x(t),(v_1^*-v_2^*)_{xtt}(t))-((z^*_1-z_2^*)''(t) \frac{x^2}{2\ell},(v_1^*-v_2^*)_{xtt}(t) )
\\ & 
\leq \|(v_1^*-v_2^*)_{x}(t)\|_{L^2(I)}
\|(v_1^*-v_2^*)_{xtt}(t)\|_{L^2(I)}
\\ & \quad +\|(k_1^*-k_2^*)\ast (v_1)_x(t)\|_{L^2(I)}\|(v_1^*-v_2^*)_{xtt}(t)\|_{L^2(I)}
\nonumber\\ & \quad +|(k_1^*-k_2^*)(t)|\|u''_0\|_{L^2(I)}\|(v_1^*-v_2^*)_{tt}(t)\|_{L^2(I)}
\\ & \quad 
+\|k_1\ast (v_1^*-v^*_2)_x(t)\|_{L^2(I)}\|(v_1^*-v_2^*)_{xtt}(t)\|_{L^2(I)}+\frac{\ell}{2}|(z^*_1-z_2^*)''(t)|\|(v_1^*-v_2^*)_{xtt}(t)\|_{L^2(I)}
\\ & 
\leq \frac {\beta}{2}\|(v_1^*-v_2^*)_{xtt }(t)\|^2_{L^2(I)}+\frac{2}{\beta}\|(v_1^*-v_2^*)_{x}(t)\|^2_{L^2(I)}+\frac{2}{\beta}\|(k_1^*-k_2^*)\ast (v_1)_x(t)\|^2_{L^2(I)}
\\ & \quad +\frac 12|(k_1^*-k_2^*)(t)|^2\|u''_0\|^2_{L^2(I)}
+\frac 12 \|(v_1^*-v_2^*)_{tt}(t)\|^2_{L^2(I)}
\\ & \quad +\frac{2}{\beta} \|k_1\ast (v_1^*-v^*_2)_x(t)\|^2_{L^2(I)}+\frac{\ell^2}{2\beta}|(z^*_1-z_2^*)''(t)|^2,
\end{align*}
for a.e. $t\in [0,\delta]$. Integrating the above inequality from $0$ to $t$, we obtain 
\begin{align*}
& \int_0^t\|(v_1^*-v_2^*)_{tt}(s)\|^2_{L^2(I)}ds+\beta\int_0^t\|(v_1^*-v_2^*)_{xtt}(s)\|^2_{L^2(I)}ds
\\ & 
\leq \|(v_1^*-v_2^*)_{xt}(0)\|^2_{L^2(I)}+\frac{4}{\beta}t\|k_1^*-k_2^*\|^2_{L^2(0,t)}\int_0^t\|(v_1)_x(s)\|^2_{L^2(I)}ds+\|k_1^*-k_2^*\|^2_{L^2(0,t)}\|u''_0\|^2_{L^2(I)}
\\ & \quad
+\frac{4}{\beta}(1+t\|k_1\|^2_{L^2(0,t)})\int_0^t\|(v_1^*-v^*_2)_x(s)\|^2_{L^2(I)}ds
+\frac {\ell^2}{\beta}\|(z_1^*-z_2^*)''\|^2_{L^2(0,t)}
\end{align*}
for all $t\in [0,\delta].$ Thus it is immediate from \eqref{eqn-4.38} and \eqref{eqn-4.39} that 
$(v_1^*-v_2^*)_{tt}\in L^2(0,\delta;H^1(I)).$

Taking the inner product with $-(v_1^*-v_2^*)_{xxt}(\cdot)$ to the equation \eqref{eqn-4.32}, we find 
\begin{align*}
& \frac 12\frac{d}{dt}\left(\|(v_1^*-v_2^*)_{xt}(t)\|^2_{L^2(I)}+\|(v_1^*-v_2^*)_{xx}(t)\|^2_{L^2(I)}
+\beta\|(v_1^*-v_2^*)_{xxt}(t)\|^2_{L^2(I)}\right)
\\ & 
=-((k_1^*-k_2^*)\ast (v_1)_{xx}(t),(v_1^*-v_2^*)_{xxt}(t))-((k_1^*-k_2^*) u''_0,(v_1^*-v_2^*)_{xxt}(t))
\\ & \quad 
-((k_1\ast (v^*_1-v^*_2)_{xx}(t),(v_1^*-v_2^*)_{xxt}(t))+((z^*_1-z_2^*)''(t) \frac{x}{\ell},(v_1^*-v_2^*)_{xxt}(t) )
\\ & 
\leq \|(k_1^*-k_2^*)\ast (v_1)_{xx}(t)\|_{L^2(I)}\|(v_1^*-v_2^*)_{xxt}(t)\|_{L^2(I)}
\nonumber\\ & \quad +|(k_1^*-k_2^*)(t)|\|u''_0\|_{L^2(I)}\|(v_1^*-v_2^*)_{xxt}(t)\|_{L^2(I)}
\\ & \quad
+\|k_1\ast (v_1^*-v^*_2)_{xx}(t)\|_{L^2(I)}\|(v_1^*-v_2^*)_{xxt}(t)\|_{L^2(I)}+|(z^*_1-z_2^*)''(t)|\|(v_1^*-v_2^*)_{xxt}(t)\|_{L^2(I)}
\\ & 
\leq 2\|(v_1^*-v_2^*)_{xxt }(t)\|^2_{L^2(I)}+\frac 12\|(k_1^*-k_2^*)\ast (v_1)_{xx}(t)\|^2_{L^2(I)}+\frac 12|(k_1^*-k_2^*)(t)|^2\|u''_0\|^2_{L^2(I)}
\\ & \quad 
+\frac 12 \|k_1\ast (v_1^*-v^*_2)_{xx}(t)\|^2_{L^2(I)}+\frac{1}{2}|(z^*_1-z_2^*)''(t)|^2,
\end{align*}
for a.e. $t\in [0,\delta]$. Integrating the above inequality from $0$ to $t$, we get 
\begin{align*}
& \|(v_1^*-v_2^*)_{xt}(t)\|^2_{L^2(I)}+\|(v_1^*-v_2^*)_{xx}(t)\|^2_{L^2(I)}
+\beta\|(v_1^*-v_2^*)_{xxt}(t)\|^2_{L^2(I)}
\nonumber\\ & 
\leq \|(v_1^*-v_2^*)_{xt}(0)\|^2_{L^2(I)}+\|(v_1^*-v_2^*)_{xx}(0)\|^2_{L^2(I)}
+\beta\|(v_1^*-v_2^*)_{xxt}(0)\|^2_{L^2(I)}
\\ & \quad
+4\int_0^t\|(v_1^*-v_2^*)_{xxt }(s)\|^2_{L^2(I)}ds+t\|k_1^*-k_2^*\|^2_{L^2(0,\delta)}\int_0^t \|(v_1)_{xx}(s)\|^2_{L^2(I)}ds
\\ & \quad +\|k_1^*-k_2^*\|^2_{L^2(0,t)}\|u''_0\|^2_{L^2(I)}
+ t\|k_1\|^2_{L^2(0,t)}\int_0^t \|(v_1^*-v^*_2)_{xx}(s)\|^2_{L^2(I)}ds+\|(z^*_1-z_2^*)''(t)\|^2_{L^2(0,t)},
\end{align*}
for all $t\in[0,\delta]$.  Using inequality \eqref{eqn-4.38} and Gronwall's inequality, we have
\begin{align*}
    & \sup\limits_{t\in [0,\delta]}\left(\|(v_1^*-v_2^*)_{xx}(t)\|^2_{L^2(I)}
+\|(v_1^*-v_2^*)_{xxt}(t)\|^2_{L^2(I)}\right)
\\ &\leq C\left(\|(v_1^*-v^*_2)_{xt}(0)\|_{L^2(I)},(\|(v_1^*-v^*_2)_{xt}(0)\|_{L^2(I)},\|(v_1^*-v^*_2)_{xxt}(0)\|_{L^2(I)}, \delta\right).
\end{align*}

Taking the inner product with $-(v_1^*-v_2^*)_{xxtt}(\cdot)$ to the equation \eqref{eqn-4.32}, we see that 
\begin{align*}
& \|(v_1^*-v_2^*)_{xtt}(t)\|^2_{L^2(I)}+\beta\|(v_1^*-v_2^*)_{xxtt}(t)\|^2_{L^2(I)}
\\ & 
=-((v_1^*-v_2^*)_{xx}(t),(v_1^*-v_2^*)_{xxtt}(t))
\\ & \quad 
-((k_1^*-k_2^*)\ast (v_1)_{xx}(t),(v_1^*-v_2^*)_{xxtt}(t))-((k_1^*-k_2^*) u''_0,(v_1^*-v_2^*)_{xxtt}(t))
\\ & \quad
-((k_1\ast (v^*_1-v^*_2)_{xx}(t),(v_1^*-v_2^*)_{xxtt}(t))+((z^*_1-z_2^*)''(t) \frac{x}{\ell},(v_1^*-v_2^*)_{xxtt}(t))
\\ & 
\leq \|(v_1^*-v_2^*)_{xx}(t)\|_{L^2(I)}\|(v_1^*-v_2^*)_{xxtt}(t)\|_{L^2(I)}
\nonumber \\ & \quad 
+\|(k_1^*-k_2^*)\ast (v_1)_{xx}(t)\|_{L^2(I)}\|(v_1^*-v_2^*)_{xxtt}(t)\|_{L^2(I)}
\\ & \quad 
+|(k_1^*-k_2^*)(t)|\|u''_0\|_{L^2(I)}\|(v_1^*-v_2^*)_{xxtt}(t)\|_{L^2(I)}
\\ & \quad 
+\|k_1\ast (v_1^*-v^*_2)_{xx}(t)\|_{L^2(I)}\|(v_1^*-v_2^*)_{xxtt}(t)\|_{L^2(I)}+|(z^*_1-z_2^*)''(t)|\|(v_1^*-v_2^*)_{xxtt}(t)\|_{L^2(I)}
\\ &  
\leq \frac {\beta}{2}\|(v_1^*-v_2^*)_{xxtt}(t)\|^2_{L^2(I)}+\frac {5}{2\beta}\|(k_1^*-k_2^*)\ast (v_1)_{xx}(t)\|^2_{L^2(I)}+\frac {5}{2\beta}|(k_1^*-k_2^*)(t)|^2\|u''_0\|^2_{L^2(I)}
\\ & \quad 
+ \frac {5}{2\beta} \|k_1\ast (v_1^*-v^*_2)_{xx}(t)\|^2_{L^2(I)}+\frac {5}{2\beta}|(z^*_1-z_2^*)''(t)|^2,
\end{align*}
for a.e. $t\in [0,\delta]$.
Integrating the above inequality from $0$ to $t$, we deduce
\begin{align*}
& \int_0^t\|(v_1^*-v_2^*)_{xtt}(s)\|^2_{L^2(I)}ds+\frac {\beta}{2}\int_0^t\|(v_1^*-v_2^*)_{xxtt}(s)\|^2_{L^2(I)}ds
\\ & 
\leq \frac {5}{2\beta}t\|k_1^*-k_2^*\|^2_{L^2(0,\delta)}\int_0^t \|(v_1)_{xx}(s)\|^2_{L^2(I)}ds
 +\frac {5}{2\beta}\|k_1^*-k_2^*\|^2_{L^2(0,t)}\|u''_0\|^2_{L^2(I)}
\\ & \quad + \frac {5}{2\beta}t\|k_1\|^2_{L^2(0,t)}\int_0^t \|(v_1^*-v^*_2)_{xx}(s)\|^2_{L^2(I)}ds+\frac {5}{2\beta}\|(z^*_1-z_2^*)''(t)\|^2_{L^2(0,t)}.
\end{align*}
Thus it is immediate that for $\delta\in (0,\tau_1\wedge (\tau-\tau_1)]$ 
\begin{align*}
    \|v_1^*-v_2^*\|_{H^2(0,\delta; H_0^1(I)\cap H^2(I))}\leq C(\delta)\|k_1^*-k_2^*\|_{L^2(0,\delta)}.
\end{align*}
Using \eqref{eqn-4.38} and Lemma \ref{lem-2.2} in the above estimate, we find
\begin{align*}
    \|v_1^*-v_2^*\|_{H^2(0,\delta; H_0^1(I)\cap H^2(I))}\leq C(\delta)\|v_1^*-v_2^*\|_{H^2(0,\delta; H_0^1(I)\cap H^2(I))}.
\end{align*}
If $ \delta$ is sufficiently small such that $\|v_1^*-v_2^*\|_{H^2(0,\delta; H_0^1(I)\cap H^2(I))}=0,$ then \eqref{eqn-4.38} ensures that $k_1^*-k_2^*$ also vanish in some neighbourhood of $0$, which contradicts the definition of $\tau_1.$ Hence, we conclude that $v_1=v_2$ and $k_1=k_2$, which completes the proof.
\end{proof}

\begin{theorem}[Global-in-time existence and uniqueness]\label{thm-4.3}
 Let the assumptions $(i1)$-$(i5)$ hold and  $T>0.$ Then the inverse problem has a unique solution
\begin{align*}
    (v,k) \in  H^2(0,T;H_0^1(I)\cap H^2(I)) \times H^1(0,T).
\end{align*}
\end{theorem}
To prove the global solvability of the inverse problem \eqref{eqn-2.2}--\eqref{eqn-wave-equiv}, we prove the result for its equivalent form, that is, for any given time $T>0,$ there exists a solution to the system \eqref{eqn-3.2}--\eqref{eqn-3.6}. The following Lemmas will be useful in the sequel.

\begin{lemma}\label{lem-4.2}
Let the assumptions $(i1)$-$(i5)$ hold. Let $\tau\in(0,T)$ and 
\begin{align*}
    (\hat v,\hat k) \in  H^2(0,\tau;H_0^1(I)\cap H^2(I)) \times H^1(0,\tau).
\end{align*}
solves the system \eqref{eqn-3.2}--\eqref{eqn-3.6} on $[0,\tau].$ Then there exists $\delta \in (0,T-\tau]$ such that the solution $(\hat v,\hat k)$ can be extended to a solution $(v,k)$ on $[0,\tau+\delta]$, such that
\begin{align*}
    (v,k) \in  H^2(0,\tau+\delta;H_0^1(I)\cap H^2(I)) \times H^1(0,\tau+\delta).
\end{align*}
\end{lemma}
\begin{proof}
From the conditions (A1)-(A7), it suffices to show that the system \eqref{eqn-4.18}--\eqref{eqn-4.22}, with $g, u_{\tau},u_{\tau \tau}$ be as defined in Lemma \ref{lem-4.1} has a solution
\begin{align*}
    (v_{\tau},k_{\tau}) \in  H^2(0,\delta;H_0^1(I)\cap H^2(I)) \times H^1(0,\delta).
\end{align*}
for some $\delta \in (0,T-\tau].$ 

 Let $\Gamma(\delta,L)$ be the space of functions
 $(\tilde v,\tilde k')\in H^2(0,\delta;H_0^1(I)\cap H^2(I))\times L^2(0,\delta)$ such that
\begin{align*}
    \|\tilde v\|_{H^2(0,\delta;H_0^1(I)\cap H^2(I))}+\|\tilde k'\|_{L^2(0,\delta)}\leq L
\end{align*}
 and \eqref{eqn-4.19}--\eqref{eqn-4.20} hold on $(0,\delta)$. The constant $L>0$ will be determined later.

We define the mapping $B: \Gamma(\delta, L)\rightarrow  \Gamma(\delta, L)$ such that $(\tilde v,\tilde k')\rightarrow  (v,k')$ through
\begin{align}\label{eqn-4.40}
& (v_{\tau})_{tt}-(v_{\tau})_{xx}-\beta (v_{\tau})_{xxtt}+\tilde k_{\tau}u''_0(x)+\int_0^t \tilde k_{\tau}(t-s)\hat v_{xx}
(x,s)\,ds
\nonumber \\ & +\int_0^t \hat k(t-s) (\tilde v_\tau)_{xx}(x,s)\,ds
+g(x,t) =z_{\tau}''\frac x\ell,\ (x,t)\in I\times [0,\delta],
\end{align}
\begin{align}\label{eqn-4.41}
    v_{\tau}\big|_{x=0}=v_{\tau}\big|_{x=\ell}=0, \quad 0\leq t\leq \delta, 
\end{align}
\begin{align}\label{eqn-4.42}
    v_{\tau}\big|_{t=0}=u_{\tau}(x), \ (v_{\tau})_t\big|_{t=0}=u_{\tau \tau}(x), \quad  x\in I,
\end{align}
with
\begin{align}\label{eqn-4.43}
    k_{\tau}'(t) & =\alpha\Bigg\{f_{\tau}^{(iv)}(t)+\int_0^\ell \tilde v_t(x,t)\varphi'''(x)dx
-k(0)\int_0^\ell  \tilde v_{\tau}(x,t)\varphi'''(x)dx
\nonumber \\ & \quad  -\int_0^t\int_0^\ell \tilde k_{\tau}'(t-s) \hat v(x,s)\varphi'''(x)dxds
-\int_0^t\int_0^\ell \hat k'(t-s) \tilde v_{\tau}(x,s)\varphi'''(x)dxds
\nonumber \\ & \quad 
-\int_t^{\tau}\int_0^\ell \hat k'(\tau+t-s)\hat v(x,s)\varphi'''(x)dxds\Bigg\}, \;\;\; t\in [0,\delta],
\end{align}
where $z_{\tau}'':=py_{\tau}'''+qy_{\tau}''$
and
\begin{align}\label{eqn-4.44}
    y_{\tau}'''(t)&:=G'[(\tilde v_{\tau})_{xx}](t)-k_{\tau}'(t)\widehat{G}[u_{xx}](0)
- k(0)G[(\tilde v_{\tau})_{xx}](t)
\nonumber \\ & \quad 
-\int_0^t\hat k'(t-s)G[(\tilde v_{\tau})_{xx}](s)\,ds-\int_0^t\tilde k_{\tau}'(t-s)G[\hat v_{xx}](s)\,ds
\nonumber \\ & \quad 
-\int_t^{\tau}\hat k'(\tau+t-s)G[\hat v_{xx}](s)\,ds, \quad  t\in [0,\delta].
\end{align}
We just need to show that the map $B: \Gamma(\delta, L)\rightarrow  \Gamma(\delta, L)$ is a contraction map. The proof constitutes similar arguments as in the proof of Theorem \ref{thm-4.1}.
\end{proof}

\begin{lemma}\label{lem-4.3}
    Let $T\in \mathbb{R}_+,\ \tau\in (0,T].$ Let $\gamma_1,\gamma_2 \in \mathbb{R}_+$ and $v\in H^2(0,\tau;H_0^1(I)\cap H^2(I))$ be such that, for all $t\in [0,\tau]$,
\begin{align}\label{eqn-4.45}
    \|v\|^2_{H^2(0,t;H_0^1(I)\cap H^2(I))}\leq \gamma_1+\gamma_2\|v_t\|^2_{L^2(0,t;H_0^1(I))}.
\end{align}
Then
\begin{align}\label{eqn-4.46}
    \|v\|_{H^2(0,t;H_0^1(I)\cap H^2(I))}\leq C\left(\gamma_1,\gamma_2,\|v_t(0)\|_{L^2(I)}, T\right),
\end{align}
where $C\left(\gamma_1,\gamma_2,\|v_t(0)\|_{L^2(I)}, T\right) \in \mathbb{R}_+$.
\end{lemma}

\begin{proof}
 Applying Gagliardo-Nirenberg's and Young's inequalities, we obtain, for all $\varepsilon \in \mathbb{R}_+,$
\begin{align}\label{eqn-4.47}
    \|v_t\|^2_{L^2(0,t;H_0^1(I))} & \leq C\|v_t\|_{L^2(0,t;L^2(I))}\|v_t\|_{L^2(0,t;H_0^1(I)\cap H^2(I))}
\nonumber \\ &
\leq C\left(\varepsilon\|v_t\|^2_{L^2(0,t;H_0^1(I)\cap H^2(I))}+C(\varepsilon)\|v_t\|^2_{L^2(0,t;L^2(I))}\right).
\end{align}
Substituting \eqref{eqn-4.47} in \eqref{eqn-4.45}, we have
\begin{align}\label{eqn-4.48}
    \|v\|^2_{H^2(0,t;H_0^1(I)\cap H^2(I))}\leq \gamma_1+\gamma_2C\varepsilon \|v_t\|^2_{L^2(0,t;H_0^1(I)\cap H^2(I))}+\gamma_2C(\varepsilon)\|v_t\|^2_{L^2(0,t;L^2(I))},
\end{align}
for every $\varepsilon\in \mathbb{R}_+.$ Choosing $\varepsilon$ in such a way that $\gamma_2C\varepsilon 
\leq \frac 12$, we get
\begin{align}\label{eqn-4.49}
    \|v\|^2_{H^2(0,t;H_0^1(I)\cap H^2(I))}\leq
2\gamma_1+C(\gamma_2)\|v_t\|^2_{L^2(0,t;L^2(I))}.
\end{align}
Since $v\in H^2(0,t;H_0^1(I)\cap H^2(I))$, $v_t$ is absolutely continuous (see \cite[Sec. 5.9, Theorem 2]{Evans_2010}).
Hence, for all $s \in (0,\tau]$,
\begin{align*}
    v_t(s)=v_t(0)+\int_0^sv_{tt}(\eta)d\eta,
\end{align*}
and 
\begin{align*}
    \|v_t(s)\|^2_{L^2(I)}\leq 2\|v_t(0)\|^2_{L^2(I)}+2\left\|\int_0^sv_{tt}(\eta)d\eta\right\|^2_{L^2(I)}.
\end{align*}
Integrating the above estimate from $0$ to $t$, we obtain
\begin{align}\label{eqn-4.50}
    \|v_t\|^2_{L^2(0,t;L^2(I))}& \leq 2t|\|v_t(0)\|^2_
{L^2(I)} + 2\int_0^t\left\|\int_0^s v_{tt}(\eta)d\eta\right\|^2_{L^2(I)}ds
\nonumber \\ &
\leq 2t\|v_t(0)\|^2_{L^2(I)} + 2
 \int_0^ts\|v_{tt}\|^2_{L^2(0,s;L^2(I))}ds.
\end{align}
If we set $z(t):= \|v_{tt}\|^2_{L^2(0,t;L^2(I))}$, from \eqref{eqn-4.49}, we get, for all $t\in [0, \tau]$,
\begin{align}\label{eqn-4.51}
    z(t) \leq 2\gamma_1 + C(\gamma_2)t\|v_t(0)\|^2_{L^2(I)} + C(\gamma_2)\int_0^tsz(s)ds.
\end{align}
An application of Gronwall's inequality in \eqref{eqn-4.51} yields
\begin{align}\label{eqn-4.52}
    z(t)\leq C (\gamma_1,\gamma_2, \|v_t(0)\|^2_{L^2(I)},T).
\end{align}
Using \eqref{eqn-4.52} in \eqref{eqn-4.50}, for all $t \in (0,\tau]$, we have
\begin{align}\label{eqn-4.53}
    \|v_t\|^2_{L^2(0,t;L^2(I))} & \leq 2t|\|v_t(0)\|^2_
{L^2(I)} + 2 \int_0^tsC (\gamma_1,\gamma_2, \|v_t(0)\|^2_{L^2(I)},T)ds
\nonumber \\ & 
\leq 2t|\|v_t(0)\|^2_
{L^2(I)} + t^2C (\gamma_1,\gamma_2, \|v_t(0)\|^2_{L^2(I)},T).
\end{align}
From \eqref{eqn-4.49} and \eqref{eqn-4.53}, we conclude
\begin{align*}
    \|v\|_{H^2(0,t;H^1_0(I)\cap H^2(I))} \leq C(\gamma_1,\gamma_2, \|v_t(0)\|^2_{L^2(I)},T),
\end{align*}
which completes the proof. 
\end{proof}

Let us now provide an a-priori estimate on $(v, k)$ defined by \eqref{eqn-4.23}. Such an estimate is of crucial importance in the proof of Theorem \ref{thm-4.3}.

\begin{lemma}\label{lem-4.4}
Let assumptions $(i1)$-$(i5)$ hold. Let
\begin{align*}
    (\hat v,\hat k) \in H^2(0,\tau;H^1_0(I) \cap H^2(I))\times H^1(0,\tau)
\end{align*}
be a solution of the system \eqref{eqn-3.2}--\eqref{eqn-3.6} in $[0, \tau]$ with $0 < \tau < T$. If
\begin{align*}
    (v, k)\in H^2(0, \tau + \delta; H^1_0(I) \cap H^2(I))\times H^1(0,\tau+\delta)
\end{align*}
is a solution of the system \eqref{eqn-3.2}--\eqref{eqn-3.6}  in $[0,\tau + \delta]$, then, we can bound $(v, k)$ and there
exists $C > 0$, such that, for all $\delta \in (0, \tau \wedge (T-\tau )]$,
\begin{align*}
    \|v\|_{H^2(0, \tau + \delta; H^1_0(I) \cap H^2(I))} + \|k'\|_{L^2(0,\tau + \delta)}\leq C.
\end{align*}
\end{lemma}

\begin{proof}
 Owing to Lemma \ref{lem-4.1} and Theorem \ref{thm-4.2}, if we set
\begin{align*}
    v_{\tau}(t) := v(\tau + t), \quad k_{\tau} (t) := k(\tau + t), \ t \in [0,\delta],
\end{align*}
we get $(v_{\tau},k_{\tau})$ solves the system \eqref{eqn-4.18}--\eqref{eqn-4.22}. Let $k = \hat k$. Applying the same arguments to estimate $v_{\tau}$ from the system \eqref{eqn-4.18}--\eqref{eqn-4.20}, \eqref{eqn-4.22} as in the proof of Theorem \ref{thm-4.2}, we obtain the following estimate:
\begin{align}\label{eqn-4.54}
    \|v_{\tau}\|^2_{H^2(0, t; H^1_0(I) \cap H^2(I))} \leq C(u_0,\hat k',\hat v, T)\left[\|u_{\tau}\|^2_{H^1_0(I) \cap H^2(I)}+\|u_{\tau\tau}\|^2_{H^1_0(I)\cap H^2(I)}+\|k'_{\tau}\|^2_{L^2(0,t)}\right],
\end{align}
for all $t\in [0,\delta].$ Similarly, we estimate $k'_{\tau}$ from the equation \eqref{eqn-4.21}
\begin{align}\label{eqn-4.55}
    \|k'_{\tau}\|^2_{L^2(0,t)} & \leq C(\hat v,\hat k',T)\left[\|f^{(iv)}\|^2_{L^2(0,t)}+\|(v_{\tau})_t\|^2_{L^2(0, t; H^1_0(I))}\right]
\nonumber \\ & 
\leq C(\hat v,\hat k',T)\left[\|f^{(iv)}\|^2_{L^2(0,t)}+\varepsilon \|(v_{\tau})_t\|^2_{L^2(0, t; H^1_0(I) \cap H^2(I))}+C(\varepsilon)\|(v_{\tau})_t\|^2_{L^2(0, t; H^1_0(I))}\right].
\end{align}
Substituting \eqref{eqn-4.55} in \eqref{eqn-4.54}, for all $t\in [0,\delta],$ we obtain
\begin{align}\label{eqn-4.56}
    \|v_{\tau}\|^2_{H^2(0, t; H^1_0(I) \cap H^2(I))} & \leq C(u_0,\hat k',\hat v, T)\Big[\|u_{\tau}\|^2_{H^1_0(I) \cap H^2(I)}+\|u_{\tau\tau}\|^2_{H^1_0(I)\cap H^2(I)} +\|f^{(iv)}\|^2_{L^2(0,t)}
\nonumber \\ & \quad
+\varepsilon \|(v_{\tau})_t\|^2_{L^2(0, t; H^1_0(I) \cap H^2(I))}+C(\varepsilon)\|(v_{\tau})_t\|^2_{L^2(0, t; H^1_0(I))}\Big].
\end{align}
Choosing $\varepsilon$ to be sufficiently small, so that $\varepsilon C(u_0,\hat k',\hat v, T)\leq \frac12,$ from \eqref{eqn-4.56}, we obtain 
\begin{align}\label{eqn-4.57}
&  \|v_{\tau}\|^2_{H^2(0, t; H^1_0(I) \cap H^2(I))} 
\nonumber\\ & \leq C(u_0,\hat k',\hat v, T)\Big[\|u_{\tau}\|^2_{H^1_0(I) \cap H^2(I)}+\|u_{\tau\tau}\|^2_{H^1_0(I)\cap H^2(I)} +\|f^{(iv)}\|^2_{L^2(0,t)}+\|(v_{\tau})_t\|^2_{L^2(0, t; H^1_0(I))}\Big].
\end{align}
Thus, we conclude from Lemma \ref{lem-4.3}, \eqref{eqn-4.57} and \eqref{eqn-4.55} that
\begin{align*}
    \|v\|_{H^2(0, \tau+\delta; H^1_0(I) \cap H^2(I))}+\|k'\|_{L^2(0,\tau+\delta)}\leq C,
\end{align*}
which completes the proof.
\end{proof}

Finally, we prove Theorem \ref{thm-4.3}.

\begin{proof}[{\bf Proof of Theorem \ref{thm-4.3}.}] 
Global-in-time uniqueness follows from Theorem \ref{thm-4.2}. It remains to prove the existence. For this, having Lemma \ref{lem-3.1} in hand, it suffices to show the existence of solution to the system \eqref{eqn-3.2}--\eqref{eqn-3.6}, say $(v,k) \in H^2(0, T; H^1_0(I) \cap H^2(I))\times H^1(0,T).$ To obtain the same, let us set
\begin{align*}
    { \bf T} & :=\Big\{\tau \in (0,T]:  (v,k)\in (v,k) \in H^2(0, T; H^1_0(I) \cap H^2(I))\times H^1(0,T)
\\
& \qquad  \qquad \mbox{ solves the system } \eqref{eqn-3.2}\text{--}\eqref{eqn-3.6}\Big\}.
\end{align*}
Obviously ${ \bf T}\neq \varnothing$ from Theorem \ref{thm-4.1}. Owing to Theorem \ref{thm-4.2}, for all $\tau \in { \bf T},$ the solution is unique and will be denoted by $(v_{\tau},k_{\tau}).$ Let $\tau_1$ and $\tau_2$ be any two elements from ${ \bf T}$ such that $\tau_1\leq \tau_2.$ Then, we have $v_{\tau_2}$ and $k_{\tau_2}$ are extensions of $v_{\tau_1}$ and $k_{\tau_1}$, respectively. Set $\widehat T:=\sup{ \bf T}$. Of course, $0< \widehat T\leq T.$ We need to prove $\widehat T=T$ for the global existence. Let us define
\begin{align*}
    v:[0,\widehat T)\rightarrow H^1_0(I) \cap H^2(I),\ v(t)=v_{\tau}(t),\ \mbox{ if } t\leq\tau,
\end{align*}
and
\begin{align*}
    k:[0,\widehat T)\rightarrow \mathbb{R},\ k(t)=k_{\tau}(t),\ \mbox{ if } t\leq\tau.
\end{align*}
Then $(v,k)$ is unique solution of the system of the system \eqref{eqn-3.2}--\eqref{eqn-3.6} in $[0,\tau),$ for all $\tau \in (0,\widehat T).$ By Lemma \ref{lem-4.4}, there exists $C>0,$ such that, for all $\tau \in [\widehat T/2,\widehat T),$
\begin{align}\label{eqn-4.58}
    \|v\|_{H^2(0, \tau; H^1_0(I) \cap H^2(I))}+\|k'\|_{L^2(0,\tau)}
 &
=\|v\|_{L^2(0, \tau; H^1_0(I) \cap H^2(I))}+\|v_t\|_{L^2(0, \tau; H^1_0(I) \cap H^2(I))}
\nonumber \\ & \quad +\|v_{tt}\|_{L^2(0, \tau; H^1_0(I) \cap H^2(I))}+\|k'\|_{L^2(0,\tau)}
\nonumber\\ & \leq C.
\end{align}
As $\tau\rightarrow\widehat T$ in \eqref{eqn-4.58}, we obtain
\begin{align*}
    \|v\|_{L^2(0, \widehat T; H^1_0(I) \cap H^2(I))}+\|v_t\|_{L^2(0,\widehat T; H^1_0(I) \cap H^2(I))}+\|v_{tt}\|_{L^2(0, \widehat T; H^1_0(I) \cap H^2(I))}+\|k'\|_{L^2(0,\widehat T)}\leq C.
\end{align*}
Therefore, we conclude that $(v,k)\in H^2(0, \widehat T; H^1_0(I) \cap H^2(I))\times H^1(0,\widehat T)$ and solves the system \eqref{eqn-3.2}--\eqref{eqn-3.6} in $[0,\widehat T].$ But this implies that $\widehat T=T.$ If not, by Lemma \ref{lem-4.2}, we could extend $(v,k)$ to a solution with domain $[0,\widehat T+\delta],$ but this is not compatible with the definition of $\widehat T,$ which completes the proof.
\end{proof}

	\begin{appendix}
	\renewcommand{\thesection}{\Alph{section}}
	\numberwithin{equation}{section}
    \section{Energy estimates for a linear problem}\label{linear-EE}

In this section, provide energy estimates for a linear system which play a key role in our analysis. Consider the following linear initial boundary value problem
\begin{equation}\label{eqn-A.1}
    \left\{
    \begin{aligned}
        v_{tt}-v_{xx}-\beta v_{xxtt} & = K, && (x,t)\in I\times (0,T),
\\
v\big|_{x=0}=v\big|_{x=\ell} & =0,  && 0<t<T,
\\
v\big|_{t=0}=v_0(x), \ v_t\big|_{t=0} & =v_1(x),   &&  x\in I.
    \end{aligned}
    \right.
\end{equation}

 Our aim is to show the following estimate:
\begin{align}\label{eqn-A.2}
    \|v\|_{H^2(0,T;H_0^1(I)\cap H^2(I))}\leq C(\|v_0\|_{H_0^1(I)\cap H^2(I)}+\|v_1\|_{H_0^1(I)\cap H^2(I))}+\|K\|_{L^2(0,T;L^2(I))}).
\end{align}
We will provide formal calculation here, one can obtain this by a standard Faedo–Galerkin approximation.

Taking the inner product with $v_t$ in $\eqref{eqn-A.1}_1$ and using integration by parts, we get
\begin{align*}
 & \frac{1}{2}\frac{d}{dt}\bigg[\|v_{t}(t)\|^2_{L^2(I)} +  \|v_{x}(t)\|^2_{L^2(I)} + \beta \|v_{xt}(t)\|^2_{L^2(I)}  \bigg]\nonumber\\ &  =  \int_0^\ell  K(x,t) v_{t}(x,t) dx 
   \leq  \|K(t)\|_{L^2(I)}\|v_{t}(t)\|_{L^2(I)}  \leq  \frac12\|K(t)\|_{L^2(I)}^2 + \frac12 \|v_{t}(t)\|^2_{L^2(I)},
\end{align*}
for a.e. $t\in[0,T]$, which gives
\begin{align}\label{eqn-A.3}
 & \sup_{t\in[0,T]}\|v_{t}(t)\|^2_{L^2(I)} +  \sup_{t\in[0,T]} \|v_{x}(t)\|^2_{L^2(I)} + \beta \sup_{t\in[0,T]} \|v_{xt}(t)\|^2_{L^2(I)}   
 \nonumber\\ & \leq  \left[\|v_{t}(0)\|^2_{L^2(I)} +    \|v_{x}(0)\|^2_{L^2(I)} + \beta  \|v_{xt}(0)\|^2_{L^2(I)} + \|K(t)\|_{L^2(0,T;L^2(I))}^2\right] e^{T}
  \nonumber\\ & = \left[ \|v_{1}\|^2_{L^2(I)} +    \|v_{0}'\|^2_{L^2(I)} + \beta  \|v_{1}'\|^2_{L^2(I)} +  \|K(t)\|_{L^2(0,T;L^2(I))}^2 \right] e^{T} =: V_1(T).
\end{align}
Taking the inner product with $v_{tt}$ in $\eqref{eqn-A.1}_1$ and using integration by parts, we get
\begin{align*}
 & \|v_{tt}(t)\|^2_{L^2(I)} + \beta \|v_{xtt}(t)\|^2_{L^2(I)}   \nonumber\\ &  = -\int_0^\ell  v_x(x,t) v_{xtt}(x,t) dx +  \int_0^\ell  K(x,t) v_{tt}(x,t) dx 
  \nonumber\\ &  \leq  \|v_{x}(t)\|_{L^2(I)}\|v_{xtt}(t)\|_{L^2(I)} +  \|K(t)\|_{L^2(I)}\|v_{tt}(t)\|_{L^2(I)}  
  \nonumber\\ & \leq \frac{1}{2\beta} \|v_{x}(t)\|^2_{L^2(I)} + \frac{\beta}{2} \|v_{xtt}(t)\|^2_{L^2(I)}+  \frac12 \|K(t)\|_{L^2(I)}^2 + \frac12 \|v_{tt}(t)\|^2_{L^2(I)},
\end{align*}
for a.e. $t\in[0,T]$, which gives
\begin{align}\label{eqn-A.4}
  \|v_{tt}(t)\|^2_{L^2(0,T;L^2(I))} + \beta \|v_{xtt}(t)\|^2_{L^2(0,T;L^2(I))}  
    & \leq \frac{1}{\beta} \|v_{x}(t)\|^2_{L^2(0,T;L^2(I))} +  \|K(t)\|_{L^2(0,T;L^2(I))}^2
  \nonumber\\ & \leq \frac{T  V_1(T)}{\beta} +  \|K(t)\|_{L^2(0,T;L^2(I))}^2< + \infty.
\end{align}
Taking the inner product with $-v_{xxt}$ in $\eqref{eqn-A.1}_1$ and using integration by parts, we get
\begin{align*}
 & \frac{1}{2}\frac{d}{dt} \bigg[ \|v_{xt}(t)\|^2_{L^2(I)} +  \|v_{xx}(t)\|^2_{L^2(I)} +  \beta \|v_{xxt}(t)\|^2_{L^2(I)} \bigg] \nonumber\\ &  = - \int_0^\ell  K(x,t) v_{xxt}(x,t) dx 
   \leq  \|K(t)\|_{L^2(I)}\|v_{xxt}(t)\|_{L^2(I)}  \leq  \frac{1}{2\beta}\|K(t)\|_{L^2(I)}^2 + \frac{\beta}{2} \|v_{xxt}(t)\|_{L^2(I)},
\end{align*}
for a.e. $t\in[0,T]$, which gives
\begin{align}\label{eqn-A.5}
 & \sup_{t\in[0,T]}\|v_{xt}(t)\|^2_{L^2(I)} +  \sup_{t\in[0,T]} \|v_{xx}(t)\|^2_{L^2(I)} + \beta \sup_{t\in[0,T]} \|v_{xxt}(t)\|^2_{L^2(I)}   
 \nonumber\\ & \leq  \|v_{xt}(0)\|^2_{L^2(I)} +    \|v_{xx}(0)\|^2_{L^2(I)} + \beta  \|v_{xxt}(0)\|^2_{L^2(I)} + \frac{1}{\beta} \|K(t)\|_{L^2(0,T;L^2(I))}^2
  \nonumber\\ & =  \|v_{1}'\|^2_{L^2(I)} +    \|v_{0}''\|^2_{L^2(I)} + \beta  \|v_{1}''\|^2_{L^2(I)} + \frac{1}{\beta} \|K(t)\|_{L^2(0,T;L^2(I))}^2 =: V_2(T).
\end{align}
Taking the inner product with $-v_{xxtt}$ and using integration by parts, we get
\begin{align*}
  \|v_{xtt}(t)\|^2_{L^2(I)} + \beta \|v_{xxtt}(t)\|^2_{L^2(I)}     &  = -\int_0^\ell  v_{xx}(x,t) v_{xxtt}(x,t) dx -  \int_0^\ell  K(x,t) v_{xxtt}(x,t) dx 
  \nonumber\\ &  \leq  \bigg[\|v_{xx}(t)\|_{L^2(I)} +  \|K(t)\|_{L^2(I)}\bigg]\|v_{xxtt}(t)\|_{L^2(I)}  
  \nonumber\\ & \leq \frac{1}{\beta} \|v_{xx}(t)\|^2_{L^2(I)} + \frac{\beta}{2} \|v_{xxtt}(t)\|^2_{L^2(I)} +  \frac{1}{\beta} \|K(t)\|_{L^2(I)}^2 ,
\end{align*}
for a.e. $t\in[0,T]$, which gives
\begin{align}\label{eqn-A.6}
 2 \|v_{xtt}(t)\|^2_{L^2(0,T;L^2(I))} + \beta \|v_{xxtt}(t)\|^2_{L^2(0,T;L^2(I))}  
    & \leq \frac{2}{\beta} \bigg[\|v_{xx}(t)\|^2_{L^2(0,T;L^2(I))} +  \|K(t)\|_{L^2(0,T;L^2(I))}^2\bigg]
  \nonumber\\ & \leq \frac{2}{\beta} \bigg[T  V_2(T) +  \|K(t)\|_{L^2(0,T;L^2(I))}^2\bigg] < +\infty.
\end{align}
Hence combining \eqref{eqn-A.3}-\eqref{eqn-A.6}, we obtain \eqref{eqn-A.2}. This completes the purpose of this section.

    \end{appendix}

\medskip\noindent
{\bf Acknowledgments:} Z.D. Totieva is funded by Ministry of Science and Higher Education of the Russian Federation, agreement № 075-02-2025-1633. K. Kinra is funded by national funds through the FCT - Funda\c c\~ao para a Ci\^encia e a Tecnologia, I.P., under the scope of the project  UIDB/00297/2020 and UIDP/00297/2020 (Center for Mathematics and Applications). Support for M. T. Mohan's research received from the National Board of Higher Mathematics (NBHM), Department of Atomic Energy, Government of India (Project No. 02011/13/2025/NBHM(R.P)/R\&D II/1137).


\begin{thebibliography}{99}


\bibitem{Bozorov_2020}  Z.R. Bozorov, Numerical determining a memory function of a horizontally-stratified elastic medium with aftereffect, \emph{Eurasian journal of mathematical and computer applications}, \textbf{8}(2)  (2020), 28--40.




\bibitem{Bukhgeim_1993}  A.L. Bukhgeim, Inverse problems of memory reconstruction, \emph{J. Inverse Ill-Posed Probl.}, {\bf 1}(3) (1993), 193--205.



\bibitem{Bukhgeim+Dyatlov_1998}   A.L. Bukhgeim and G.V. Dyatlov, Inverse problems for equations with memory, in {\it Inverse problems and related topics (Kobe, 1998)}, 19--35, Chapman \& Hall/CRC Res. Notes Math., 419, Chapman \& Hall/CRC, Boca Raton, FL. 





\bibitem{Cavaterra+Grasselli_1994}   C. Cavaterra and M. Grasselli, Identifying memory kernels in linear thermoviscoelasticity of Boltzmann type, \emph{Math. Models Methods Appl. Sci.} {\bf 4}(6) (1994), 807--842. 





\bibitem{Colombo+Lorenzi_1998_ADE} F. Colombo and A. Lorenzi, An inverse problem in the theory of combustion of materials with memory, \emph{Adv. Differential Equations}, {\bf 3}(1) (1998), 133--154.  

\bibitem{Colombo+Lorenzi_1997_JMAA} F. Colombo and A. Lorenzi, Identification of time and space dependent relaxation kernels for materials with memory related to cylindrical domains. I, II, \emph{J. Math. Anal. Appl.}, {\bf 213}(1) (1997), 32--62, 63--90. 





\bibitem{Colombo_2001_DIE} F. Colombo, A quasilinear parabolic inverse problem in H\"older spaces, \emph{Differential Integral Equations}, {\bf 14}(10) (2001), 1237--1258.  

\bibitem{Colombo_2001_JMAA} F. Colombo, Direct and inverse problems for a phase-field model with memory, \emph{J. Math. Anal. Appl.}, {\bf 260}(2) (2001), 517--545.   





\bibitem{Colombo+Guidetti+Lorenzi_2003_AMSA} F. Colombo, D. Guidetti and A. Lorenzi, Integro-differential identification problems for the heat equation in cylindrical domains, \emph{Adv. Math. Sci. Appl.}, {\bf 13}(2) (2003), 639--662. 

\bibitem{Colombo+Guidetti+Lorenzi_2003_DSA} F. Colombo, D. Guidetti and A. Lorenzi, Integrodifferential identification problems for thermal materials with memory in non-smooth plane domains, \emph{Dynam. Systems Appl.}, {\bf 12}(3--4) (2003),  533--559.  

\bibitem{Colombo+Guidetti_2002_ZAA} F. Colombo and D. Guidetti, A unified approach to nonlinear integro-differential inverse problems of parabolic type, \emph{Z. Anal. Anwendungen}, {\bf 21}(2) (2002),  431--464.  





\bibitem{Colombo_2005} F. Colombo, Quasilinear systems with memory kernels, \emph{J. Math. Anal. Appl.}, {\bf 311}(1) (2005), 40--68. 


\bibitem{Colombo_2007_PD} F. Colombo, An inverse problem for a parabolic integrodifferential model in the theory of combustion, \emph{Phys. D}, {\bf 236}(2) (2007), 81--89. 

\bibitem{Colombo_2007_Nonlinearity} F. Colombo, An inverse problem for the strongly damped wave equation with memory, \emph{Nonlinearity}, {\bf 20}(3) (2007), 659--683. 






\bibitem{Colombo+Guidetti_2007}  F. Colombo and D. Guidetti, A global in time existence and uniqueness result for a semilinear integrodifferential parabolic inverse problem in Sobolev spaces, \emph{Math. Models Methods Appl. Sci.}, {\bf 17}(4) (2007), 537--565. 












\bibitem{Durdiev_1994}  D.K. Durdiev, A multi-dimensional inverse problem for an equation with memory, \emph{Siberian Math. J.}, {\bf 35}(3) (1994), 514--521; translated from \emph{Sibirsk. Mat. Zh.}, {\bf 35}(3) (1994), 574--582, {\rm ii}. 





\bibitem{Durdiev_2020} U.D. Durdiev, Numerical method for determining the dependence of the dielectric permittivity on the frequency in the equation of electrodynamics with memory, \emph{Sibirskie Èlektronnye Matematicheskie Izvestiya [Siberian Electronic Mathematical Reports]}, \textbf{17}  (2020), 179--189.



\bibitem{Durdiev+Totieva_2023}   D.K. Durdiev and Z.D. Totieva, {\it Kernel determination problems in hyperbolic integro-differential equations}, Infosys Science Foundation Series Infosys Science Foundation Series in Mathematical Sciences, Springer, Singapore, [2023]. 



\bibitem{Evans_2010} L.C. Evans, {\it Partial differential equations}, second edition,  Graduate Studies in Mathematics, 19, Amer. Math. Soc., Providence, RI, 2010.









\bibitem{Grasselli+Kabanikhin+Lorenzi_1992}  M. Grasselli, S.I. Kabanikhin and A. Lorenzi, An inverse problem for an integro-differential equation, \emph{Siberian Math. J.}, {\bf 33}(3) (1992), 415--426; translated from \emph{Sibirsk. Mat. Zh.}, {\bf 33}(3) (1992), 58--68, 218. 





\bibitem{Janno+Wolfersdorf_2001}  J. Janno and L. von Wolfersdorf, An inverse problem for identification of a time- and space-dependent memory kernel in viscoelasticity, \emph{Inverse Problems}, {\bf 17}(1) (2001), 13--24. 




\bibitem{Kaltenbacher+Khristenko+Nikolic+Rajendran+Wohlmuth_2022}  B. Kaltenbacher, U. Khristenko, V. Nikoli\'c, M.L. Rajendran and B. Wohlmuth, Determining kernels in linear viscoelasticity, \emph{J. Comput. Phys.}, {\bf 464} (2022), Paper No. 111331, 23 pp.



\bibitem{Karchevsky+Fatianov_2001}    A.L. Karchevsky and A.G.E. Fatianov, Numerical solution of the inverse problem for a system of elasticity with the aftereffect for a vertically inhomogeneous medium, \emph{Sibirskii Zhurnal Vychislitel'noi Matematiki}, \textbf{4}(3) (2001), 259--268.






\bibitem{Kumar+Kinra+Mohan_2021}  P. Kumar, K. Kinra and M.T. Mohan, A local in time existence and uniqueness result of an inverse problem for the Kelvin-Voigt fluids, \emph{Inverse Problems}, {\bf 37}(8) (2021), Paper No. 085005, 34 pp.






   \bibitem{Li+Xi_2019} F. Li and S. Xi, Dynamic properties of a nonlinear viscoelastic Kirchhoff-type equation with acoustic control boundary conditions. I, \emph{Mathematical Notes}, \textbf{106}(5) (2019), pp.814--832.








\bibitem{Lorenzi+Messina+Romanov_2007}  A. Lorenzi, F. Messina and V.G. Romanov, Recovering a Lam\'e{} kernel in a viscoelastic system, \emph{Appl. Anal.}, {\bf 86}(11) (2007), 1375--1395.



\bibitem{Lorenzi+Romanov_2011} A. Lorenzi and V.G. Romanov, Recovering two Lam\'e{} kernels in a viscoelastic system, \emph{Inverse Probl. Imaging}, {\bf 5}(2) (2011), 431--464.





   \bibitem{Park+Park_2011} J.Y. Park and S.H. Park, Decay rate estimates for wave equations of memory type with acoustic boundary conditions, \emph{Nonlinear Anal.}, {\bf 74} (2011), no.~3, 993--998.







   \bibitem{Park+Kang_2008} J.Y. Park and J.-R. Kang, Existence, uniqueness and uniform decay for the non-linear degenerate equation with memory condition at the boundary, \emph{Appl. Math. Comput.} {\bf 202}(2) (2008), 481--488.



	\bibitem{POV} A. I. Prilepko, D. G. Orlovsky and I. A. Vasin, \emph{Methods for Solving Inverse Problems in Mathematical Physics}, Marcel Dekker Inc., New York, 2000.
		
	





\bibitem{Romanov_2012}  V.G. Romanov, A two-dimensional inverse problem for the viscoelasticity equation, \emph{Sibirsk. Mat. Zh.} \textbf{53}(6) (2012), 1401--1412; translation in
\emph{Sib. Math. J.} \textbf{53}(6) (2012), 1128--1138.




\bibitem{Romanov_2012_JAIM}  V.G. Romanov, Stability estimates for the solution in the problem of determining the kernel of the viscoelasticity equation, \emph{Sib. Zh. Ind. Mat.}, \textbf{15}(1) (2012), 86--98; translation in \emph{J. Appl. Ind. Math.} \textbf{6}(3) (2012), 360--370.




\bibitem{Romanov_2014}  V.G. Romanov, Inverse problems for equations with a memory, \emph{Eurasian Journal of Mathematical and Computer Applications}, \textbf{2}(4) (2014), 51--80.





\bibitem{Romanov+Yamamoto_2010}  V.G. Romanov and M. Yamamoto, Recovering a Lam\'e{} kernel in a viscoelastic equation by a single boundary measurement, \emph{Appl. Anal.}, {\bf 89})(3) (2010), 377--390.





\bibitem{Santos_2001}  M.L. Santos, Asymptotic behavior of solutions to wave equations with a memory condition at the boundary, \emph{Electron. J. Differential Equations}, {\bf 2001}, No. 73, 11 pp.







\bibitem{Totieva+Durdiev_2024} Z.D. Totieva and D.K. Durdiev, Inverse problem for wave equation of memory type with acoustic boundary conditions, \emph{Eurasian Journal of Mathematical and Computer Applications}, \textbf{12}(3) (2024), 173--188.








\bibitem{Totieva_2024}   Z.D. Totieva, A global in time existence and uniqueness result of a multi-dimensional inverse problem, \emph{Appl. Anal.}, {\bf 103}(4) (2024), 701--723. 





\bibitem{Vicente_2009}  A. Vicente, Wave equation with acoustic/memory boundary conditions, \emph{Bol. Soc. Parana. Mat. (3)}, {\bf 27}(1) (2009), 29--39.



\bibitem{Whitham_1999}  G.B. Whitham, {\it Linear and nonlinear waves}, reprint of the 1974 original, Pure and Applied Mathematics (New York) A Wiley-Interscience Publication, , Wiley, New York, 1999. 


	\bibitem{WL}
		B. Wu and J. Liu, A global in time existence and uniqueness result for an integrodifferential hyperbolic inverse problem with memory effect, \emph{J. Math. Anal. Appl.}, \textbf{343} (2011), 585-604.
































   
\end{thebibliography}
\end{document}